\documentclass[11pt,reqno]{amsart}
\usepackage{fullpage}
\usepackage{dsfont}
\usepackage{amsmath}
\usepackage{amsfonts}
\usepackage{amssymb}
\usepackage{graphicx}
\usepackage{mathrsfs}
\usepackage{epsfig}
\usepackage{amsthm}
\usepackage{verbatim}
\usepackage{graphicx}
\usepackage{color}
\usepackage{url}

\usepackage[colorlinks,citecolor=black,linkcolor=black,
            bookmarksopen,
            bookmarksnumbered
           ]{hyperref}

\newcommand{\beq}{\begin{equation}}
\newcommand{\eeq}{\end{equation}}
\newcommand{\e}{\mathrm{e}}

\newcommand{\nab}{\langle \nabla \rangle_c}
\newcommand{\Lc}{\mathcal{L}_c}

\newtheorem{theorem}{Theorem}[section]
\newtheorem{defi}[theorem]{Definition}
\newtheorem{lemma}[theorem]{Lemma}
\newtheorem{cor}[theorem]{Corollary}

\newtheorem{rem}[theorem]{Remark}

\allowdisplaybreaks[3]

\author{María Cabrera Calvo}
\address{LJLL (UMR 7598), Sorbonne Universit\'e, UPMC, 4 place Jussieu, 75005, Paris, France (M. Cabrera Calvo)}
\email{maria.cabrera\_calvo@sorbonne-universite.fr}

\author{Katharina Schratz}
\address{LJLL (UMR 7598), Sorbonne Universit\'e, UPMC, 4 place Jussieu, 75005, Paris, France (K. Schratz)}
\email{katharina.schratz@sorbonne-universite.fr}
 
\begin{document}
\begin{abstract}
We propose a novel class of  uniformly accurate integrators for the Klein--Gordon equation  which capture classical $c=1$ as well as highly-oscillatory non-relativistic regimes $c\gg1$ and, at the same time, allow for  low regularity approximations. In particular, the schemes converge with order $\tau$ and $\tau^2$, respectively, under lower regularity assumptions than classical schemes, such as splitting or exponential integrator methods, require.  The new schemes in addition preserve the nonlinear Schr\"odinger (NLS) limit on the discrete level. More precisely, we will design our schemes in such a way that in the limit $c\to \infty$ they converge to a recently introduced class of low regularity integrators for NLS.
\end{abstract}

\title[]{Uniformly accurate low regularity integrators for the Klein--Gordon equation from the classical to non-relativistic limit regime}

\maketitle

\section{Introduction}
We consider the nonlinear Klein--Gordon equation
\begin{equation}
\begin{aligned}\label{eq:kgr}
& c^{-2} \partial_{tt} z - \Delta z + c^2 z = \vert z\vert^{2} z, \quad z(0,x) = z_0(x),\quad \partial_t z(0,x) = c^2 z'_0(x)\\
\end{aligned}
\end{equation}
which in the so-called non-relativistic limit regime $c\to \infty$ collapses to the classical cubic nonlinear Schr\"odinger equation. More precisely, the exact solution $z$ of \eqref{eq:kgr} allows (for sufficiently smooth data) the expansion 
\begin{align}\label{zexp}
z(t,x) = \frac{1}{2}\left( \e^{ic^2 t} u_{\ast,\infty}(t,x) + \e^{-ic^2t} \overline{v}_{\ast,\infty}(t,x) \right) +  \mathcal{O}(c^{-2})
\end{align}
 on a time-interval uniform in $c$, where $(u_{\ast,\infty}, v_{\ast,\infty})$ satisfy the cubic Schr\"odinger limit system
\begin{equation}\label{NLSlimit}
\begin{aligned}
i \partial_t u_{\ast, \infty} &=& \frac{1}{2} \Delta u_{\ast,\infty} + \frac{1}{8}\big(\left \vert  u_{\ast,\infty}\right\vert^2 + 2\left\vert v_{\ast,\infty}\right\vert^2\big)  u_{\ast,\infty}\qquad u_{\ast,\infty}(0) = \varphi - i \gamma
\\
i \partial_t v_{\ast, \infty} &=& \frac{1}{2} \Delta v_{\ast,\infty} + \frac{1}{8}\big(\left \vert  v_{\ast,\infty}\right\vert^2 + 2\left\vert u_{\ast,\infty}\right\vert^2\big)  v_{\ast,\infty},\qquad v_{\ast,\infty}(0) = \overline{\varphi} - i \overline{\gamma}
\end{aligned}
\end{equation}
with initial values
\begin{align*}
z(0,x) \stackrel{c\to\infty}{\longrightarrow} {\varphi(x)} \quad \text{and}\quad c^{-1}\left(c^2-\Delta\right)^{-1/2} \partial_t z(0,x) \stackrel{c\to\infty}{\longrightarrow} {\gamma(x)},
\end{align*}
see \cite[Formula (1.3)]{MaNak02} and for the periodic setting \cite[Formula (37)]{FS13}. 

Reproducing this limit behaviour, and in general non relativistic regimes of large $c$,  on the discrete level  is  highly challenging as the large parameter $c$ triggers  oscillations  of type
$$
\left(\e^{\ell ic\sqrt{c^2-\Delta} t}\right)_{\ell \in \mathbb{Z}}
$$
which are difficult to resolve numerically without imposing severe step size restrictions at the cost of huge computational costs. We refer to \cite{EFHI09,HLW} and the references therein for an introduction and overview on highly-oscillatory problems. Gautschi-type methods (see \cite{HoLu99}) were first analyzed in~\cite{BG} and introduce a global error of order $c^4 \tau^2$ which requires the CFL-type condition $c^2 \tau <1$. To overcome this difficulty so-called limit integrators which  reduce the highly-oscillatory problem to  the corresponding non-oscillatory limit system  (i.e., $c\to \infty$ in \eqref{eq:kgr}) as well as uniformly accurate schemes based on multiscale expansions were introduced in \cite{FS13} and \cite{BaoZ,BFS17,ChC} for \emph{smooth solutions}  (at least in $H^{6}$).

{ We also refer to \cite{BaoLT} for recent results on improved error bounds on time splitting methods for long-time dynamics of the nonlinear Klein--Gordon equation with weak nonlinearity in the relativistic $c=1$ regime.}

 In the very recent work \cite{RS}, a low regularity approximation technique  for a class of abstract evolution equations, including parabolic as well as wave type systems, was introduced. This new approach in general allows us to resolve the time dynamics of PDEs under lower regularity assumptions (in space) than classical methods, such as splitting or exponential integrator methods,  require.  
 
Up to now it was an open question whether one can couple the idea of low regularity integrators in space and low regularity (or more precisely uniformly accurate) integrators   in time. The main difficulty lies in  controlling - on the discrete level -  the underlying oscillations triggered by the leading operator
\begin{equation}\label{cnab}
c\nab= c \sqrt{- \Delta + c^2}.
\end{equation}
 This is much more involved than in previous works \cite{BFS17,RS} where either (i) $c=1$ such that no time oscillations appear or (ii)  all regimes of~$c$ are captured, however, $-\Delta$ is considered to be ``neglectable" (as smooth solutions are imposed) and the coupled oscillations $e^{i t c\nab}$ can be simply expanded into a Taylor series 
 \begin{equation}\label{TayK1}
 e^{i t c\nab} = e^{i t c^2} + \mathcal{O}(t \Delta)
 \end{equation}
  reducing the full spectrum of high frequencies to only one single  oscillation $e^{i t c^2}$.

The aim of this paper lies in closing this gap: We will develop a new class of  uniformly accurate schemes for the Klein--Gordon equation \eqref{eq:kgr} which resolve the time oscillations in $c$ and at the same time allow  for a low regularity approximation in space (in the spirit of  \cite{RS}).  The central idea to achieve this   lies in embedding the full spectrum of  oscillations -- in space and in time --
into the numerical discretisation. The new schemes will in particular  
\begin{itemize}
\item[(I)]   allow convergence for rougher data than  splitting or exponential integrator methods \\(involving only first instead of second order derivatives in the local error)
\item[(II)]    converge uniformly in $c$ with error bounds independent of $c$ (allowing us to capture non-oscillatory classical $c=1$ up to highly-oscillatory non-relativistic regimes $c\gg1$), and
\item[(III)]   preserve the NLS limit \eqref{NLSlimit} on the discrete level.
\end{itemize}
More precisely, we will design our schemes in such a way that in the limit $c\to \infty$ we recover  a low regularity integrator discretisation of NLS (see, e.g., \cite{OS18}). 

Compared to previously proposed uniformly accurate schemes (such as \cite{BFS17}) and the  new regularity framework  introduced for non oscillatory problems (\cite{RS}) the construction is in our setting much  more delicate as the full spectrum of frequencies triggered by $-\Delta$ 
is coupled with the possible large parameter $c$. The key  step lies  in   suitable two- and three-scale expansions and    the essentiel estimates on the commutator structure of  \eqref{cnab} (see  Lemma \ref{bound:comm}). This will allow us, in the construction and error analysis of the new schemes, to establish estimates that hold uniformly in $c$. \\

\noindent{\bf Outline of the paper.} We will motivate the new first order scheme \eqref{scheme} in Section~\ref{sec:first} and its second order counterpart \eqref{scheme2} in Section \ref{sec:second}. We state their uniform  and asymptotic convergence in Theorem \ref{thm:glob}, Theorem \ref{thm:asymp}, Theorem \ref{thm:glob2} and Remark \ref{thm:asymp2}, respectively. Our ideas can be extended to higher order methods. Numerical experiments in Section \ref{sec:num} underline our theoretical findings.\\

\noindent{\bf Notation.}
In the following we fix $r>d/2$. { We will} assume periodic boundary conditions that is $x\in \mathbb{T}^d$.   { Our ansatz can, however, be extended to bounded domains $x \in \Omega \subset \mathbb{R}^d$ (equipped with suitable boundary conditions) and the full space $x \in \mathbb{R}^d$ with   Fourier  analysis on $\mathbb{R}^d$ and suitable extension techniques on bounded domains. For the possible implementation on bounded domains we refer to finite difference and finite element methods, see, also \cite{BaoFD}, and for the full space setting to  Malmquist–Takenaka functions (see, e.g.,  \cite{IKSW}).} For reasons of clarity of presentation we will sometimes use the $\mathcal{O}$ notation. We stress that we only employ this notation with constants independent of $c$, i.e., we say that
$$
v- w = \mathcal{O}(z) \quad \text{ if }\quad  \Vert v-w  \Vert_r \leq k \Vert z\Vert_r
$$
for some constant $k$ that can be chosen independently of $c$. 
\section{Formulation as a first order system}
For a given $c > 0$, we define the operator
\begin{align}\label{nab}
\nab = \sqrt{- \Delta + c^2}
\end{align}
which as a Fourier multiplier takes the form
$
(\nab)_k = \sqrt{k^2 + c^2}.
$
With this notation at hand,  we can rewrite the Klein--Gordon equation  \eqref{eq:kgr}  as a first-order system in time. For this purpose we set
\begin{equation}\label{eq:uv}
u = z - i c^{-1}\nab^{-1} \partial_t z , \qquad v = \overline z - i c^{-1}\nab^{-1} \partial_t \overline z
\end{equation}
such that in particular
\begin{equation}\label{eq:zuv}
z = \frac12 (u + \overline{v}).
\end{equation}
\begin{rem}\label{rem:realz}
If $z$ is real, then $u \equiv v$. 
\end{rem}
A short calculation shows that in terms of the variables $u$ and $v$ equation \eqref{eq:kgr} reads
\begin{equation}
\label{eq:NLSc}
\begin{array}{rcl}
i \partial_t u &=& -c\nab u + c\nab^{-1} f(\textstyle \frac12 (u + \overline v)), \\[2ex]
i \partial_t v &=& -c\nab v  + c\nab^{-1} f(\textstyle \frac12 (\overline u +  v))
\end{array}
\end{equation}
with the nonlinearity
\[
 f(z) = \vert z \vert^2 z.
 \]
and initial conditions (see \eqref{eq:kgr})
\begin{equation}
\label{eq:BCc}
u(0) = z(0) -ic^{-1}\nab^{-1} z'(0) , \quad \mbox{and}\quad v(0) =\overline{z(0)} -ic^{-1}\nab^{-1} \overline{z'(0)}. 
\end{equation}
Now we are in the position to derive the first order low regularity uniformly accurate scheme. To present the main ideas more clearly, in the following we will restrict our attention to the real case $z(t,x) \in \mathbb{R}$ such that by Remark \ref{rem:realz} we have that $u = v$. However, the proposed construction and analysis can be easily extended to the complex setting.
\section{A first order low regularity uniformly accurate integrator}\label{sec:first}
Duhamel's formula for \eqref{eq:NLSc} reads
\begin{equation}\label{Duh}
u(t) = e^{i t c\nab}u(0)- i \frac18 c\nab^{-1}  e^{i t c\nab}\int_0^t   e^{-i s c\nab}\left( u(s) + \overline{u}(s) \right)^3 ds
\end{equation}
and we are left with approximating the above integral. \emph{Naively} we could apply a  straightforward Taylor series expansion of the solution
$$
u(s)  = u(0) + s\cdot \partial_t u(0) + \ldots.
$$
This would, however, introduce powers of $c^2 $ due to the observation that
$$
\partial_t u(t) = c\nab u(t) + \text{ lower order terms,} \quad c \nab = c^2 -\frac12 \Delta + \ldots.
$$
In order to overcome this and allow for uniform convergence in $c$  at low regularity we will iterate Duhamel's formula \eqref{Duh} in the spirit of
\begin{equation}\label{o1}
u(s)  = e^{i s c\nab}u(0) +  \mathcal{R}_1(s,u),
\end{equation}
where the remainder can be bounded uniformly in $c$ and does not require any spatial regularity of the solution. More precisely,  thanks to the observations that
\begin{equation}\label{stabi}
\Vert c\nab^{-1}\Vert_r \leq 1 \qquad\Vert e^{i t c\nab}\Vert_r  =1\quad \forall t \in \mathbb{R}
\end{equation}
it holds that
\begin{equation}\label{R1}
\begin{aligned}
\Vert  \mathcal{R}_1(s,u)\Vert_r &= \left \Vert
 \frac18 c\nab^{-1}  e^{i s c\nab}\int_0^s  e^{-i s_1 c\nab}\left( u(s_1) + \overline{u}(s_1) \right)^3 ds_1
\right\Vert_r  \leq s K(\sup_{0\leq t\leq s} \Vert u(t)\Vert_r).
\end{aligned}
\end{equation}
Plugging the expansion \eqref{o1} into \eqref{Duh} thus yields that
\begin{equation}\label{uexp1}
\begin{aligned}
u(t) & = e^{i t c\nab}u(0)- i \frac18 c\nab^{-1}  e^{i t c\nab}\int_0^t   e^{-i s c\nab}\left(e^{i s c\nab}u(0) + e^{-i s c\nab} \overline{u}(0) \right)^3 ds\\
&+ \mathcal{R}_2(t,u)\\
\end{aligned}
\end{equation}
where the remainder $ \mathcal{R}_2(t,u)$ satisfies the bound
\begin{equation}\label{R2}
\begin{aligned}
\Vert  \mathcal{R}_2(t,u)\Vert_r & \leq t^2 K(\sup_{0\leq \overline t \leq t} \Vert u(\overline t)\Vert_r).
\end{aligned}
\end{equation}
Hence, we are left with deriving a low regularity uniformly accurate approximation to the central oscillatory integral 
\begin{equation}\label{osc}
\mathcal{I}(t, c\nab, u(0)) :=\int_0^t   e^{-i s c\nab}\left(e^{i s c\nab}u(0) + e^{-i s c\nab} \overline{u}(0) \right)^3 ds.
\end{equation}
For this purpose we introduce the operator
$$
\Lc = c\nab -c^2.
$$
With this notation at hand the principal oscillations \eqref{osc} take the form
\begin{equation}\label{osc2}
\mathcal{I}(t, c\nab, u(0)) =\int_0^t   e^{-i s (\Lc+c^2)}\left(e^{i s (\Lc+c^2)}u(0) + e^{-i s(\Lc+c^2)} \overline{u}(0) \right)^3 ds.
\end{equation}

Note that the operator $\Lc$ can be (for smooth solutions) uniformly bounded with respect to $c$. More precisely, it holds that
\begin{equation}\label{bound:Lc}
\Vert \Lc v \Vert_r  = \left\Vert \left(c \sqrt{-\Delta +c^2}-c^2\right) v\right\Vert_r \leq  \frac12 \Vert v \Vert_{r+2}.
\end{equation}
In contrast to previous works we, however, will not base our schemes on the above estimate. In contrast, it will be essential in our low regularity approximations to also embed the oscillations ${e}^{i t \Lc}$ (triggered by $-\Delta$) into our numerical discretisation and not simply neglect them by  a Taylor series expansion
$$
{e}^{i t \Lc} = 1 + \mathcal{O}(t \Delta)
$$
 in the spirit of \eqref{TayK1}.

In order to capture all oscillations we will  tackle each oscillatory integral in \eqref{osc2}  separately
\begin{align}\label{I1}
\mathcal{I}(t, c\nab, u(0)) &=\int_0^t   e^{-i s (c^2+\Lc)}\left(e^{i s (c^2+\Lc)}u(0)\right)^3 ds\\\label{I2}
& + 3  \int_0^t   e^{-i s  (c^2+\Lc)}\left[\left(e^{i s  (c^2+\Lc)}u(0)\right)^2  e^{-i s  (c^2+\Lc)} \overline{u}(0)  \right]ds\\\label{I3}
& + 3  \int_0^t   e^{-i s  (c^2+\Lc)}\left[\left(e^{-i s  (c^2+\Lc)} \overline{u}(0)\right)^2  e^{i s  (c^2+\Lc)}  {u}(0)  \right]ds\\\label{I4}
& +\int_0^t   e^{-i s  (c^2+\Lc)}\left(e^{-i s (c^2+\Lc)}\overline u(0)\right)^3 ds.
\end{align}
Beforehand we introduce a commutator type definition which will be essential in our local error estimates.
\begin{defi}\label{def:comm}
Let us define for a function  $H(v_{1}, \cdots v_{n})$, $n \geq 1$ and the linear operator $\Lc$
the commutator type term 
 $ \mathcal{C}[H, \Lc]$    by
\begin{align}\label{DlowIntro}
\mathcal{C}[H, \Lc](v_{1}, \cdots, v_{n})
= 
- \Lc ( H \left( v_{1}, \cdots    v_{n}\right))
+  \sum_{i=1}^n   \,  D_{i} H( v_{1}, \cdots,  v_{n}) \cdot  \Lc v_{i}
\end{align}
where $D_{i} H$ stands for the partial differential of $H$ with respect to the variable $v_{i}.$ We will also make use of the iterated commutator 
$$ \mathcal{C}^2[ H, \Lc ] (v_{1}, \cdots, v_{n})
 = \mathcal{C}[ \mathcal{C}[H, \Lc], \Lc] (v_{1}, \cdots,  v_{n}).$$
 Furthermore, we set
 $$f_{\text{cub}}(v,w,z) = v w z.
 $$
\end{defi}

\begin{lemma}[Bound on the commutator]\label{bound:comm}
We  have that
\begin{align*}
&\Vert \mathcal{C}[f_{\text{cub}}(\cdot,\cdot,\cdot), \Lc](v,w,z)\Vert_r  \leq k_1 \Vert v\Vert_{r+1} \Vert w\Vert_{r+1}
 \Vert z\Vert_{r+1}\\
& \Vert \mathcal{C}^2[f_{\text{cub}}(\cdot,\cdot,\cdot), \Lc](v,w,z)\Vert_r  \leq k_2 \Vert v\Vert_{r+2} \Vert w\Vert_{r+2}
 \Vert z\Vert_{r+2}.
 \end{align*}
 for constants $k_1$, $k_2>0$ which can be chosen independently of $c$.
\end{lemma}
\begin{proof}
We will show that for $f_{\text{quad}}(v,w) =v w$ have that
\begin{equation}\label{quad}
 \mathcal{C}[f_{\text{quad}}(\cdot,\cdot), \Lc](v,w) =   \mathcal{O}\left( \nabla v \nabla w \right)\end{equation}
 {
and
\begin{equation}\label{quad2}
\mathcal{C}^2[f_{\text{quad}}(\cdot,\cdot), \Lc](v,w) =   \mathcal{O}\left( \Delta v \Delta w \right).
\end{equation} }
The assertion{ s} for $f_{\text{cub}}$ then follow the line of argumentation. 

Note that
\begin{equation}\label{C1}
 \mathcal{C}[f_{\text{quad}}(\cdot,\cdot), \Lc](v,w) =  -  ( c\nab - c^2) (v w) + w  ( c\nab - c^2) v + v   ( c\nab - c^2)  w.
\end{equation}
{
In the following we use for $k, l \in \mathbb{Z}^d$  the notation
$$
k l = k_1 l_1 + \ldots + k_d l_d \quad \text{and}\quad \vert k\vert^2 = k_1^2 + \ldots + k_d^2.
$$
}
Then we observe that in Fourier space we obtain
\begin{align*}
( c\nab - c^2) (v w) & = c^2\left(\sqrt{1-\frac{\Delta}{c^2}}-1\right)(vw)= \sum_{k,l} e^{i k x} \hat{v}_{k-l} \hat{w}_l  c^2\left(\sqrt{1+\frac{k^2}{c^2}}-1\right)\\
& =  \sum_{ k,l \in \mathbb{Z}^d} e^{i k x}   \hat{v}_{k-l} \hat{w}_l  \frac{k^2}{ \sqrt{1+\frac{k^2}{c^2}}+1}
\end{align*}
such that by \eqref{C1}  we have
\begin{equation}\label{C2}
\begin{aligned}
 \mathcal{C}[f_{\text{quad}}(\cdot,\cdot), \Lc](v,w) & = -
 \sum_{ k,l \in \mathbb{Z}^d} e^{i k x}   \hat{v}_{k-l} \hat{w}_l  \left(  \frac{k^2}{ \sqrt{1+\frac{k^2}{c^2}}+1}
 -
 \frac{(k-l)^2}{ \sqrt{1+\frac{(k-l)^2}{c^2}}+1}
 -
 \frac{l^2}{ \sqrt{1+\frac{l^2}{c^2}}+1} \right) \\
 & =  \sum_{ k,l \in \mathbb{Z}^d}e^{i k x}   \hat{v}_{k-l} \hat{w}_l  \Big( T_1(k) - T_2(k-l) - T_3(l) \Big) .
\end{aligned}
\end{equation}
Without loss of generality let us assume that $ \vert l\vert < \vert k - l\vert$. Then we can bound the third term $T_3$  in~\eqref{C2} as follows
\begin{equation}\label{T3}
 \left \vert \sum_{ k,l \in \mathbb{Z}^d} e^{i k x}   \hat{v}_{k-l} \hat{w}_l  T_3(l) \right\vert
 = \left \vert  \sum_{ k,l \in \mathbb{Z}^d} e^{i k x}   \hat{v}_{k-l} \hat{w}_l   
 \frac{l^2}{ \sqrt{1+\frac{l^2}{c^2}}+1} \right \vert <
 \sum_{ k,l \in \mathbb{Z}^d}\vert   \hat{v}_{k-l} \vert \vert k-l\vert \vert \hat{w}_l   \vert \vert l \vert
\end{equation}
such that only first order derivatives are required and it remains to establish a suitable bound on the difference $T_1(k) - T_2(k-l)$  in \eqref{C2}. For this purpose we set
$$
D_1(k) =  { \sqrt{1+\frac{k^2}{c^2}}+1}, \quad D_{ 2}(k-l) = { \sqrt{1+\frac{(k-l)^2}{c^2}}+1}.
$$
Then we have that
\begin{equation*}
\begin{aligned}
T_1(k) - T_2(k-l) = & \frac{k^2}{ \sqrt{1+\frac{k^2}{c^2}}+1}
 -
 \frac{(k-l)^2}{ \sqrt{1+\frac{(k-l)^2}{c^2}}+1} \\&  =
 \frac{k^2 \left( \sqrt{1+\frac{(k-l)^2}{c^2}}+1\right)
 -
(k-l)^2\left( \sqrt{1+\frac{k^2}{c^2}}+1\right)
}{D_1(k) D_2(k-l)}\\
& =
 \frac{k^2 \left( \sqrt{1+\frac{(k-l)^2}{c^2}} -  \sqrt{1+\frac{k^2}{c^2}}\right)
 { +  2k l  \left( \sqrt{1+\frac{k^2}{c^2}}+1\right)
 -}
l^2 \left( \sqrt{1+\frac{k^2}{c^2}}+1\right)
}{D_1(k) D_2(k-l)}.
 \end{aligned}
\end{equation*}
As we assume that $ \vert l\vert < \vert k - l\vert$ we can bound the second and third term similarly as in  \eqref{T3} and it remains to bound the first term
\begin{equation*}
\begin{aligned}
T_0(k,k-l) =  \frac{k^2  \left( \sqrt{1+\frac{(k-l)^2}{c^2}} -  \sqrt{1+\frac{k^2}{c^2}}\right) }{D_1(k) D_2(k-l)}.
 \end{aligned}
\end{equation*}
We note that
\begin{equation*}
\begin{aligned}
 T_0(k,k-l)  & =  \frac{k^2  \left( \left(1+\frac{(k-l)^2}{c^2}\right) -  \left(1+\frac{k^2}{c^2}\right)\right) }{D_1(k) D_2(k-l)\left( D_1(k) +  D_2(k-l) { -2} \right) }  =  \frac{ \frac{k^2}{c^2}   \left( - 2kl +l^2 \right)  }{D_1(k) D_2(k-l)\left( D_1(k) +  D_2(k-l){ -2} \right) } . \end{aligned}
\end{equation*}
Next we use that
$$
\left \vert \frac{1}{D_1(k) D_2(k-l)\left( D_1(k) +  D_2(k-l) { -2} \right) }\right\vert \leq 
\left \vert \frac{1}{D_1(k)D_1(k) }\right\vert\leq \frac{c^2}{k^2}.
$$
Hence, { thanks to $ \vert l\vert < \vert k - l\vert$ we obtain}  that
\begin{equation*}
\begin{aligned}
\vert  T_0(k,k-l)  \vert & \leq   \left\vert - 2kl +l^2 \right\vert \leq 3 \vert { k-l} \vert \vert l \vert
 \end{aligned}
\end{equation*}
and we can conclude similar to \eqref{T3}. Therefore, we obtain that
$$
\Vert \mathcal{C}[f_{\text{quad}}(\cdot,\cdot), \Lc](v,w)\Vert_r  \leq k_1 \Vert v\Vert_{r+1} \Vert w\Vert_{r+1}
$$
for some constant $k_1$ independent of $c$. 

{ 
For the second assertion \eqref{quad2} we observe by using Definition \ref{def:comm} that
\begin{equation*}
\begin{aligned}
\mathcal{C}^2[f_{\text{quad}}(\cdot,\cdot), \Lc](v,w)&=\Lc^2 \big[ vw \big]- 2\Lc\big[ (\Lc v)w+(\Lc w)v\big] +2\big[\Lc v\big]\big[\Lc w\big] + \big[ \Lc^2 v \big]w+ \big[ \Lc^2 w \big]v \\
&=  \sum_{ k,l \in \mathbb{Z}^d} e^{ikx} \hat{v}_{k-l}\hat{w}_l M(k,l,k-l),
\end{aligned}
\end{equation*}
where
\begin{align*}
M(k,l) & = \bigg[\frac{k^2}{D(k)} \bigg]^2 - 2 \frac{k^2}{D(k)}
\bigg[\frac{l^2}{D(l)}+ \frac{(k-l)^2}{D(k-l)}\bigg] + 2 \frac{l^2}{D(l)} \frac{(k-l)^2}{D(k-l)}+  \bigg[\frac{l^2}{D(l)} \bigg]^2  +  \bigg[\frac{(k-l)^2}{D(k-l)} \bigg]^2 
\end{align*}
with
$$
D(j) = \sqrt{1+\frac{j}{c^2}}+1\quad \text{for } j\in\mathbb{Z}^d.
$$
Rearranging the terms, i.e., using that
$$
 2 \frac{l^2}{D(l)} \frac{(k-l)^2}{D(k-l)}+  \bigg[\frac{l^2}{D(l)} \bigg]^2  +  \bigg[\frac{(k-l)^2}{D(k-l)} \bigg]^2 =   \bigg[\frac{l^2}{D(l)}  +  \frac{(k-l)^2}{D(k-l)} \bigg]^2 
 $$
  we find that
\begin{align*}
M(k,l) & = 
\bigg[\frac{k^2}{D(k)} \bigg]^2 - 2 \frac{k^2}{D(k)}
\bigg[\frac{l^2}{D(l)}+ \frac{(k-l)^2}{D(k-l)}\bigg]  +  \bigg[\frac{l^2}{D(l)}  +  \frac{(k-l)^2}{D(k-l)} \bigg]^2 
\\ & =\bigg[ \frac{k^2}{D(k)}-\frac{l^2}{D(l)}-\frac{(k-l)^2}{D(k-l)} \bigg]^2.
\end{align*}
%
Thus, we can conclude, by the proof of the first assertion, see also \eqref{C2} that
$$
\|\mathcal{C}^2[f_{\text{quad}}(\cdot,\cdot), \Lc](v,w)\|_r\leq k_2\|v\|_{r+2}\|w\|_{r+2}
$$
for some constant $k_2$ independent of $c$.
}
\end{proof}

Now we are in the position to develop suitable uniformly accurate low regularity approximations to the central oscillations 
$
\mathcal{I}(t, c\nab, u(0)) 
$ given in \eqref{I1} - \eqref{I4}. We will make use of the $\varphi_1$ function defined by $ \varphi_1(\xi) = \frac{e^\xi-1}{\xi}$.

\begin{lemma}[Approximation of the integral \eqref{I1}]\label{lem:I1}
It holds that
$$
\int_0^t   e^{-i s (c^2+\Lc)}\left(e^{i s (c^2+\Lc)}v\right)^3 ds = t \varphi_1(2i c^2t) v^3
+ \mathcal{O}\left( t^2 \mathcal{C}[f_{\text{cub}}(\cdot,\cdot,\cdot), \Lc](v,v,v)  \right).
$$
\end{lemma}
\begin{proof}
We note that
\begin{align*}
\mathcal{I}_1(t, c\nab, v) := \int_0^t    e^{-i s (c^2+\Lc)}\left(e^{i s (c^2+\Lc)}v\right)^3 ds   = \int_0^t e^{2i c^2s}  e^{-i s \Lc}
\left(e^{i s  \Lc}v\right)^3 ds.
\end{align*}
Next we define the filtered function
$$
\mathcal{N}(s_1,\Lc,v) =  e^{-i s_1 \Lc}
\left(e^{i s_1  \Lc}v\right)^3
$$
which allows us to express the oscillatory integral as follows
\begin{align*}
\mathcal{I}_1(t, c\nab, v) 
& = \int_0^t e^{2i c^2s}   \mathcal{N}(s,\Lc,v)  ds.
\end{align*}
Taylor series expansion of $\mathcal{N}(s_1,\Lc,v) $ around $s_1 = 0$ yields by noting that
$$
  \mathcal{N}(0,\Lc,v) = v^3
$$
the following expansion
\begin{align*}
\mathcal{I}_1(t, c\nab, v)   & = \int_0^t e^{2i c^2s}  \left( v^3 + \int_0^s \partial_{s_1}  \mathcal{N}(s_1,\Lc,v)  ds_1\right) ds.
\end{align*}
The assertion thus follows from the observation that
$$
\partial_{s_1} \mathcal{N}(s_1,\Lc,v)   = - i \Lc  e^{-i s_1 \Lc}
\left(e^{i s_1  \Lc}v\right)^3 + 3i  e^{-i s_1 \Lc} \left(e^{i s_1  \Lc}v\right)^2 { \left(e^{i s_1  \Lc}\Lc v\right)}
$$
which implies 
$$
\partial_{s_1} \mathcal{N}(s_1,\Lc,v)   =   e^{-i s_1 \Lc} 
\mathcal{C}[f_{\text{cub}} (\cdot,\cdot,\cdot), i\Lc]  \left(e^{i s_1  \Lc} v,e^{i s_1  \Lc} v,e^{i s_1  \Lc} v\right) .
$$
\end{proof}

\begin{lemma}[Approximation of the integral \eqref{I2}]\label{lem:I2}
It holds that
\begin{multline*}
3  \int_0^t   e^{-i s  (c^2+\Lc)}\left[\left(e^{i s  (c^2+\Lc)}v\right)^2  e^{-i s  (c^2+\Lc)} \overline{v}\right]ds \\ = 3 t v^2  \varphi_1(-2i t \Lc) \overline v + \mathcal{O}\left( t^2 \mathcal{C}[f_{\text{cub}}(\cdot,\cdot,\cdot), \Lc](v,v,\overline v)  \right).
\end{multline*}
\end{lemma}
\begin{proof}
We note that
\begin{align*}
\int_0^t  & e^{-i s  (c^2+\Lc)}\left[\left(e^{i s  (c^2+\Lc)}v\right)^2  e^{-i s  (c^2+\Lc)} \overline{v}  \right]ds =
\int_0^t   e^{-i s   \Lc}\left[\left(e^{i s   \Lc}v\right)^2  e^{-i s   \Lc} \overline{v}  \right]ds\\
& = \int_0^t \left( v^2 e^{-2 i s   \Lc} \overline{v}  + \int_0^s \partial_{s_1} \mathcal{N}(s, s_1,\Lc,v) ds_1\right) ds,
\end{align*}
where we have introduced the filtered function
\begin{equation}\label{filter2}
\mathcal{N}(s, s_1,\Lc,v) =  e^{-i s_1   \Lc}\left[\left(e^{i s_1   \Lc}v \right)^2  e^{ i s_1   \Lc} e^{-2 i s   \Lc} \overline{v}  \right] .
\end{equation}
The assertion thus follows from the observation that
$$
\partial_{s_1} \mathcal{N}(s, s_1,\Lc,v)  = e^{-i s_1   \Lc} 
\mathcal{C}[f_{\text{cub}} (\cdot,\cdot,\cdot), i\Lc] \left(e^{i s_1  \Lc} v,e^{i s_1  \Lc} v, e^{ i s_1   \Lc} e^{-2 i s   \Lc}  \overline v\right).
$$
\end{proof}
\begin{lemma}[Approximation of the integral \eqref{I3}]\label{lem:I3}
It holds that
\begin{multline*}
 3  \int_0^t   e^{-i s  (c^2+\Lc)}\left[\left(e^{-i s  (c^2+\Lc)} \overline{v}\right)^2  e^{i s  (c^2+\Lc)}  v \right]ds  \\ = 3  t     v \varphi_1({- 2 i t c\nab })\overline{v}^2  +  \mathcal{O}\left( t^2 \mathcal{C}[f_{\text{cub}}(\cdot,\cdot,\cdot), \Lc](v,\overline v,\overline v)  \right).
\end{multline*}
\end{lemma}
\begin{proof}
Similarly to above we define a filtered function, where we now need a three scale formulation due to the quadratic term $\overline u^2$,
$$
\mathcal{N}(s, s_1,s_2, \Lc,v) =  e^{-i s_1   \Lc}\left[\left(  e^{ i s_1   \Lc}  {v} \right)e^{- 2i s   \Lc}   e^{ i s_1   \Lc}e^{ i s_2   \Lc} \left(e^{-i s_2   \Lc}\overline v \right)^2 \right].
$$
Taylor series expansion around $s_1=0$ and $s_2 = 0$ yields with the observation
$$
\mathcal{N}(s, 0,0, \Lc,v)= v  e^{- 2i s   \Lc} \overline v^2
$$
 that
\begin{align*}  \int_0^t   & e^{-i s  (c^2+\Lc)}\left[\left(e^{-i s  (c^2+\Lc)} \overline{v}\right)^2  e^{i s  (c^2+\Lc)}  v \right]ds = \int_0^t e^{-2 i s  c^2} \mathcal{N}(s, s,s, \Lc,v)ds \\
& = \int_0^t e^{-2 i s  c^2} \left( \mathcal{N}(s, 0 ,s , \Lc,v)+ \int_0^{s } \partial_{\xi } \mathcal{N}(s, \xi  ,s , \Lc,v)d\xi \right) ds\\
& = \int_0^t e^{-2 i s  c^2} \left( \mathcal{N}(s, 0 ,0 , \Lc,v)+ \int_0^{s } \partial_{s_1} \mathcal{N}(s, s_1  ,s , \Lc,v)d\xi + \int_0^{s } \partial_{s_2 } \mathcal{N}(s, 0 ,s_2, \Lc,v)d\xi  \right) ds\\
& = \int_0^t e^{-2 i s  c^2} \left( v  e^{- 2i s   \Lc} \overline v^2 + \int_0^{s } \partial_{s_1} \mathcal{N}(s, s_1  ,s , \Lc,v) ds_1 + \int_0^{s } \partial_{s_2} \mathcal{N}(s, 0, s_2  , \Lc,v)  d s_2\right) ds\\
& = \int_0^t v  e^{- 2i s   (\Lc+c^2)} \overline v^2  ds +\int_0^t \left(  \int_0^{s } \partial_{s_1} \mathcal{N}(s, s_1  ,s , \Lc,v) ds_1 + \int_0^{s } \partial_{s_2} \mathcal{N}(s, 0  ,s_2, \Lc,v)  d s_2  \right) ds.
  \end{align*}
The assertion follows by noting that $\Lc + c^2 = c\nab$ together with the definition of the $\varphi_1$ function and commutator $ \mathcal{C}$.
\end{proof}

\begin{lemma}[Approximation of the integral \eqref{I4}]\label{lem:I4}
It holds that
\begin{multline*}
 \int_0^t   e^{-i s  (c^2+\Lc)}\left(e^{-i s (c^2+\Lc)}\overline v \right)^3 ds   =   t      \varphi_1({- 2 i t ( c\nab+c^2) })\overline{v}^3  +  \mathcal{O}\left( t^2 \mathcal{C}[f_{\text{cub}}(\cdot,\cdot,\cdot), \Lc](\overline v,\overline v,\overline v)  \right).
\end{multline*}
\end{lemma}
\begin{proof}
Similarly to above we define the filtered function
$$
\mathcal{N}(s, s_1, \Lc,v) =  e^{i s_1   \Lc}  e^{- 2i s   \Lc}   \left(e^{-i s_1   \Lc}\overline v \right)^3 
$$
such that
\begin{align*}
 \int_0^t    e^{-i s  (c^2+\Lc)}\left(e^{-i s (c^2+\Lc)}\overline v \right)^3 ds
 & =  \int_0^t  e^{-4 i s  c^2}\mathcal{N}(s, s, \Lc,v)ds\\
 & =   \int_0^t  e^{-4 i s  c^2}\left( \mathcal{N}(s, 0, \Lc,v)+ \int_0^{s} \partial_{s_1} \mathcal{N}(s, s_1, \Lc,v) d s_1 \right) ds\\
  & =   \int_0^t  e^{-4 i s  c^2}  e^{- 2i s   \Lc} \overline v^3ds + \int_0^t    \int_0^{s} \partial_{s_1} \mathcal{N}(s, s_1, \Lc,v) d s_1  ds
 \end{align*}

The assertion follows by noting that $2\Lc + 4 c^2 = 2 (c\nab+c^2)$.
\end{proof}
 Lemma \ref{lem:I1} - \ref{lem:I4} together with the commutator bound given in Lemma \ref{bound:comm} allow us to obtain the following corollary on a uniformly accurate low regularity expansion of the underlying oscillations.
\begin{cor}\label{cor:osc}
The oscillations $\mathcal{I}(t, c\nab, v)$ defined in \eqref{osc} allow the  expansion
\begin{align*}
\mathcal{I}(t, c\nab, v) & =  t \varphi_1(2i c^2t) v^3 +  3 t v^2  \varphi_1(-2i t \Lc) \overline v
+  3  t     v \varphi_1({- 2 i t c\nab })\overline{v}^2+ t      \varphi_1({- 2 i t ( c\nab+c^2) })\overline{v}^3  \\
& + \mathcal{O}\left( t^2 (\nabla v)^3 \right).
\end{align*}
\end{cor}
Corollary \ref{cor:osc} together with the expansion of the exact solution given in \eqref{uexp1}  motivates the first order uniformly accurate low regularity integrator
\begin{equation}\label{scheme}
\begin{aligned}
u^{n+1}  = e^{i \tau  c\nab}u^n - i  \tau  \frac18 c\nab^{-1}  e^{i \tau c\nab}
&\Big[ \varphi_1(2i c^2\tau) (u^n)^3 +  3 (u^n)^2 \varphi_1(-2i \tau  \Lc) \overline{u^n}
\\ &+  3      u^n \varphi_1({- 2 i \tau c\nab })(\overline{u^n})^2+    \varphi_1({- 2 i \tau ( c\nab+c^2) })(\overline{u^n})^3 \Big].
\end{aligned}
\end{equation}
In the next sections we will carry out the error and asymptotic  analysis of the above scheme.
\subsection{Local error analysis}\label{sec:loc}
We start with the local error analysis. For this purpose we will denote by $\varphi^t$ the exact flow of   \eqref{eq:kgr} and by $\Phi^\tau$ the numerical flow defined by the scheme \eqref{scheme}, such that
$$
u(t_n+\tau) = \varphi^\tau(u(t_n) ) \quad \text{and}\quad u^{n+1} = \Phi^\tau(u^n).
$$
\begin{lemma}\label{thm:loc}
Fix $r>d/2$. The local error  $ \varphi^\tau(u(t_n))- \Phi^\tau(u(t_n))$ satisfies 
$$
  \varphi^\tau(u(t_n))- \Phi^\tau(u(t_n)) = \mathcal{O}\left( \tau^2 \mathcal{C}[f_{\text{cub}}(\cdot,\cdot,\cdot), \Lc](v_1,v_2,v_3)(t)\right)
$$
for $v_j(t) \in \{ u(t), \overline u(t)\}$.
\end{lemma}
\begin{proof}
The assertion follows by the expansion of the exact solution given in \eqref{uexp1} together with the error bound \eqref{R2} and  Corollary \ref{cor:osc}.
\end{proof}
\subsection{Stability analysis}\label{sec:stab}
\begin{lemma} \label{thm:stab} Fix  $r>d/2$. The numerical flow defined by the scheme \eqref{scheme} is stable in $H^r$ in the sense that for two functions $v,w\in H^{r}$  we have that 
$$
\Vert \Phi^\tau(v)- \Phi^\tau(w) \Vert_r  \leq  e^{\tau L} \Vert v - w\Vert_r$$
where the constant $L$ depends on the $H^{r}$ norm of $v$ and  $w$.
\end{lemma}
\begin{proof} The assertion follows thanks to the estimates in \eqref{stabi} together with the fact that  $\vert \varphi_1(i \xi) \vert \leq 1$ for all $\xi \in \mathbb{R}$.
\end{proof}
\subsection{Global error}\label{sec:glob}
\begin{theorem} \label{thm:glob} Fix  $r>d/2$ and assume that the solution of \eqref{eq:kgr} satisfies $u \in \mathcal{C}([0,T];H^{r+1})$.  Then there exists a $\tau_0>0$ such that for all $0<\tau \leq \tau_0$  the following global error estimate holds for  $u^n$ defined in \eqref{scheme}
\begin{align*}
\Vert u(t_n)- u^n \Vert_{r} &  \leq    \tau  K\left(\sup_{0\leq t\leq t_{n}} \Vert u(t) \Vert_{r+1}\right) ,
\end{align*}
where $K$ depends on $t_n$ and the $H^{r+1}$ norm of the solution $u$, but can be chosen independently of $c$.

\end{theorem}
\begin{proof} The proof follows by the local error estimate in Lemma \ref{thm:loc} together with the stability estimate in Lemma \ref{thm:stab} by a Lady Windamere's fan argument (\cite{HLW}).
\end{proof}{ 
\begin{rem}
Note that by  \eqref{eq:BCc} we have that $u(0) \in H^{r+1}$ if
\begin{equation}\label{condz}
z(0) \in H^{r+1}, \quad c^{-1}\nab^{-1} z'(0)\in H^{r+1}.
\end{equation}
Hence, if the solution $z$ of \eqref{eq:kgr} satisfies \eqref{condz} we can conclude by  local wellposedness of \eqref{eq:NLSc} that there exists a $T=T_{r+1}>0$ such that $u \in \mathcal{C}([0,T_{r+1}];H^{r+1})$.
\end{rem}}
\subsection{Asymptotic convergence to low regularity NLS limit integrator}
Note that for  $c\to \infty $  we formally observe that
$$
 c\nab^{-1} \to 1, \quad {\Lc} \to -\frac12 \Delta, \quad  \tau \varphi_1({- m i \tau c\nab })\to 0 \text{ for }  m \neq 0
$$
such that our scheme \eqref{scheme} formally reduces to 
\begin{equation*}
\begin{aligned}
u^{n+1}  \to  e^{i c^2 \tau} e^{- i \tau \frac12 \Delta }\left[u^n - i  \tau  \frac38    (u^n)^2 \varphi_1(i \tau  \Delta ) \overline{u^n} \right].
\end{aligned}
\end{equation*}
The latter is exactly the low regularity integrator for the NLS limit system \eqref{NLSlimit} originally  proposed in \cite{OS18}. In the following we will establish the precise asymptotic approximation result.

Let us denote by $u_{\ast, \infty}^n$ the low regularity NLS integrator defined by the sequence  (cf. \cite{OS18}) 
\begin{equation}\label{nls:scheme}
u_{\ast, \infty}^{n+1} = e^{- i \tau \frac12 \Delta }\left[u_{\ast, \infty}^n - i  \tau  \frac38    (u_{\ast, \infty}^n)^2 \varphi_1(i \tau  \Delta ) \overline{u_{\ast, \infty}^n}\right].
\end{equation}
Then we obtain the following asymptotic convergence of the uniformly accurate low regularity integrator $u^n$ (defined in \eqref{scheme}) for the Klein--Gordon equation \eqref{eq:kgr} towards the low regularity NLS integrator $u_{\ast, \infty}^n$ (defined in \eqref{nls:scheme}) which approximates the NLS equation \eqref{NLSlimit}.
\begin{theorem}\label{thm:asymp}
 Fix  $r>d/2$ and assume that the solution of \eqref{eq:kgr} satisfies $u \in \mathcal{C}([0,T];H^{r+3+\varepsilon})$ for some $\varepsilon>0$.  Then  there exists a $\tau_0>0$ such that for all $\tau < \tau_0$ the asymptotic error estimate holds for  $u^n$ defined in \eqref{scheme} and $u_{\ast, \infty}^n$ defined in \eqref{nls:scheme}
\begin{align*}
\Vert  u^n - e^{i c^2 t_n} u_{\ast, \infty}^n \Vert_{H^r} &  \leq    {c}^{-1}K\big(\sup_{0\leq t \leq t_n} \Vert u(t)\Vert_{r+3+\varepsilon}\big) \end{align*}
where $K$ depends on  the $H^{r+3+\varepsilon}$ norm of the solution $u$, but can be chosen independently of $c$.
\end{theorem}
\begin{proof}
First we note that for  $\tau \leq  \frac{1}{c }$ we have, by the asymptotic approximation on the continuous level which holds at order $\frac{1}{c}$  (cf. \eqref{zexp}) as well as the first order time convergence of the scheme $u^n$ towards Klein--Gordon and $u_{\ast, \infty}^{n}$ towards NLS,  that
\begin{align*}
\left \Vert   u^{n+1} - e^{i c^2 t_{n+1}} u_{\ast, \infty}^{n+1} \right \Vert_{H^r} 
& \leq \left  \Vert u(t_{n+1}) - e^{i c^2 t_{n+1}} u_{\ast, \infty}(t_{n+1}) \right \Vert_{H^r} 
\\& + \left \Vert   u^{n+1} -u(t_{n+1}) \right \Vert_{H^r} +  \left \Vert    u_{\ast, \infty}^{n+1}  -u{\ast, \infty}(t_{n+1}) \right \Vert_{H^r} 
\\& \leq k\left ( \tau +  \frac{1}{c }  \right) \leq k  \frac{1}{c}.
\end{align*}

Next let us assume that $\tau > \frac{1}{c  }$. Taking the difference of \eqref{scheme} and  \eqref{nls:scheme} (the latter multiplied with the oscillatory phase $e^{i c^2 t_{n+1}}$  (cf. \eqref{zexp})) we see thanks to \eqref{stabi} that
\begin{equation}\label{uuast}
\begin{aligned}
\Vert  & u^{n+1} - e^{i c^2 t_{n+1}} u_{\ast, \infty}^{n+1} \Vert_{H^r}  \leq
\left\Vert e^{i \tau  c\nab}  u^n - e^{i c^2\tau} e^{- i \tau \frac12 \Delta } e^{i c^2 t_{n}} u_{\ast, \infty}^{n} 
\right\Vert_r\\
&+\tau  \frac38   \left \Vert e^{i \tau  c\nab}   c\nab^{-1}  \Big[(u^n)^2 \varphi_1(-2i \tau  \Lc) \overline{u^n}
\Big]-     e^{i c^2 \tau} e^{- i \tau \frac12 \Delta }  \Big[ (e^{i c^2 t_n }  u_{\ast, \infty}^n)^2 \varphi_1(i \tau  \Delta ) e^{-i c^2 t_{n}}  \overline{u_{\ast, \infty}^n}\Big]
\right\Vert_r \\
&
+  \tau  \left \Vert
\varphi_1(2i c^2\tau) (u^n)^3\right\Vert_r+  3  \tau    \left \Vert  u^n \varphi_1({- 2 i \tau c\nab })(\overline{u^n})^2\right\Vert_r + \tau   \left \Vert \varphi_1({- 2 i \tau ( c\nab+c^2) })(\overline{u^n})^3\right\Vert_r.
\end{aligned}
\end{equation}
{
First we use that 
$$
c\nab^{-1}   = \frac{1}{\sqrt{1-\frac{\Delta}{c^2}}} = 1 + \mathcal{O}\left( \frac{\Delta}{c^2}\right)
$$
such that
$$
\Vert \left( c \nab^{-1} - 1\right) v \Vert_r \leq   \frac{k}{c^2} \Vert v\Vert_{r+2}
$$
and hence
\begin{align*}
& \tau  \frac38   \left \Vert e^{i \tau  c\nab}   c\nab^{-1}  \Big[(u^n)^2 \varphi_1(-2i \tau  \Lc) \overline{u^n}
\Big]-     e^{i c^2 \tau} e^{- i \tau \frac12 \Delta }  \Big[ (e^{i c^2 t_n }  u_{\ast, \infty}^n)^2 \varphi_1(i \tau  \Delta ) e^{-i c^2 t_{n}}  \overline{u_{\ast, \infty}^n}\Big]
\right\Vert_r \\
& \leq   \tau  \frac38   \left \Vert e^{i \tau  c\nab}   \left(c\nab^{-1}  -1\right) \Big[(u^n)^2 \varphi_1(-2i \tau  \Lc) \overline{u^n}
\Big]
\right\Vert_r\\
& +  \tau  \frac38   \left \Vert e^{i \tau  c\nab}    \Big[(u^n)^2 \varphi_1(-2i \tau  \Lc) \overline{u^n}
\Big]-     e^{i c^2 \tau} e^{- i \tau \frac12 \Delta }  \Big[ (e^{i c^2 t_n }  u_{\ast, \infty}^n)^2 \varphi_1(i \tau  \Delta ) e^{-i c^2 t_{n}}  \overline{u_{\ast, \infty}^n}\Big]
\right\Vert_r\\
& \leq   k \frac{\tau}{c^2}  \Vert {u^n}\Vert_r +  \tau  \frac38   \left \Vert e^{i \tau  c\nab}    \Big[(u^n)^2 \varphi_1(-2i \tau  \Lc) \overline{u^n}
\Big]-     e^{i c^2 \tau} e^{- i \tau \frac12 \Delta }  \Big[ (e^{i c^2 t_n }  u_{\ast, \infty}^n)^2 \varphi_1(i \tau  \Delta ) e^{-i c^2 t_{n}}  \overline{u_{\ast, \infty}^n}\Big]
\right\Vert_r.
\end{align*}

}
Newt we use that for any $0\neq \xi \in \mathbb{R}$ it holds that
$$
\tau \Vert \varphi_1(i c^2\tau  \xi) \Vert_r = \left\Vert \frac{e^{i c^2\tau \xi}-1}{i c^2 \xi} \right\Vert_r\leq \frac{2}{c^2\vert \xi\vert}.
$$
Furthermore, by the expansion
$$
c\nab = c^2 - \frac12 \Delta + \mathcal{O}\left( \frac{\Delta^{1+\alpha}}{c^{2\alpha}} \right) \quad \text{for } 0\leq \alpha \leq 1
$$
we see (by choosing $\alpha = \frac12$) that
$$
\left \Vert \left( e^{i \tau  c\nab}   - e^{i c^2\tau} e^{- i \tau \frac12 \Delta }\right)  v\right\Vert_r \leq  k \frac{\tau}{c}\Vert v \Vert_{r+3}.
$$
Finally let us note that
\begin{align*}
\tau  \varphi_1(-2i \tau  \Lc) - \tau \varphi_1(i \tau  \Delta ) &  = \int_0^\tau \left( e^{-2i s  \Lc} - e^{i s \Delta}\right)ds= \int_0^\tau   e^{i s \Delta} \left( { e^{-2i s(\Lc-\frac12\Delta)}} -1 \right)ds
\end{align*}
which implies that
\begin{align*}
\tau  \left\Vert \left(\varphi_1(-2i \tau  \Lc) -  \varphi_1(i \tau  \Delta ) \right) v\right\Vert_r \leq  \frac{ \tau  }{c }\Vert v \Vert_{r+3}.
\end{align*}
Applying the above estimates in \eqref{uuast} we obtain that
\begin{equation}\label{uuast2} 
\begin{aligned}
\left \Vert   u^{n+1} - e^{i c^2 t_{n+1}} u_{\ast, \infty}^{n+1} \right\Vert_{H^r}&  \leq
(1+\tau k_0) \left\Vert  u^n -  e^{i c^2 t_{n}} u_{\ast, \infty}^{n} 
\right\Vert_r+ \tau  \frac{1}{c}k_1\big(\Vert u_n \Vert_{r+3}\big)  + c^{-2} k_2\left(\Vert u^n\Vert_r\right),
\end{aligned}
\end{equation}
where the constant $k_0$ depends on $\Vert u_n\Vert_r$ and $\Vert  u_{\ast, \infty}^{n} 
 \Vert_r$,  $k_1$ depends on $\Vert u_n \Vert_{r+3}$ and $k_2$ on $\Vert u^n\Vert_r$, but both constants can be chosen independently of $c$. Iterating \eqref{uuast2}  we obtain   that  
 \begin{align*}
\left \Vert   u^{n+1} - e^{i c^2 t_{n+1}} u_{\ast, \infty}^{n+1} \right\Vert_{H^r}&  \leq
 \frac{1}{c }K_1\big(\sup_{0\leq k \leq n} \Vert u^k \Vert_{r+3}\big)  + \frac{1}{c^2\tau}  K_2\left(\sup_{0\leq k \leq n}\Vert u^k\Vert_r\right),
\end{align*}
where $K_1$, $K_2$ can be chosen independently of $c$. As we assume that $\tau >  \frac{1}{c}$ we can conclude that
$$
\left \Vert   u^{n+1} - e^{i c^2 t_{n+1}} u_{\ast, \infty}^{n+1} \right\Vert_{H^r}  \leq  \frac{1}{c} K_1\big(\sup_{0\leq k \leq n} \Vert u^k \Vert_{r+3}\big) .
$$
Note that the regularity assumption  $u \in \mathcal{C}([0,T];H^{r+3+\varepsilon})$ implies a priori the boundedness of the numerical solution in $H^{r+3}$, i.e.,  that $\sup_{0\leq k \leq n} \Vert u^k \Vert_{r+3} < +\infty$ by the global convergence result in Theorem \ref{thm:glob}. This yields the  assertion.
\end{proof}
{ 
\begin{rem}
On the theoretical level one chan show asymptotic convergence (without or lower rate) under much lower regularity assumptions, see, e.g. \cite[Theorem 1.1]{MaNak02}.
\end{rem}
}
\section{A second order low regularity uniformly accurate integrator}\label{sec:second}
Iterating Duhamel's formula for \eqref{eq:NLSc} we obtain
\begin{equation}\label{Duh2}
\begin{aligned}
u(t) 
& = e^{i t c\nab}u(0)- i \frac18 c\nab^{-1}  e^{i t c\nab}\mathcal{I}(t,c\nab, u(0))\\
&- 3i \frac18 c\nab^{-1}  e^{i t c\nab}\int_0^t   e^{-i s c\nab}
\left(\Big(e^{i s c\nab}u(0)\Big)^2  ( -i \frac18 c\nab^{-1})e^{i s c\nab}\mathcal{I}(s,c\nab, u(0)) \right)ds\\
&- 3i \frac18 c\nab^{-1}  e^{i t c\nab}\int_0^t   e^{-i s c\nab}
\left(\Big(e^{i s c\nab}u(0)\Big)^2  ( i \frac18 c\nab^{-1})e^{-i s c\nab} \overline{\mathcal{I}(s,c\nab, u(0))} \right)ds\\
&- 3i \frac18 c\nab^{-1}  e^{i t c\nab}\int_0^t   e^{-i s c\nab}
\left(\Big(e^{-i s c\nab}\overline{u(0)}\Big)^2  ( -i \frac18 c\nab^{-1})e^{i s c\nab}\mathcal{I}(s,c\nab, u(0)) \right)ds\\
&- 3i \frac18 c\nab^{-1}  e^{i t c\nab}\int_0^t   e^{-i s c\nab}
\left(\Big(e^{-i s c\nab}\overline{u(0)}\Big)^2  ( i \frac18 c\nab^{-1}) e^{-i s c\nab}\overline{\mathcal{I}(s,c\nab, u(0))} \right)ds\\
&-6 i \frac18 c\nab^{-1}  e^{i t c\nab}\int_0^t   e^{-i s c\nab}
\left(\Big\vert e^{i s c\nab}u(0)\Big\vert^2  (- i \frac18 c\nab^{-1})  e^{i s c\nab} \mathcal{I}(s,c\nab, u(0)) \right)ds\\
&-6 i \frac18 c\nab^{-1}  e^{i t c\nab}\int_0^t   e^{-i s c\nab}
\left(\Big\vert e^{i s c\nab}u(0)\Big\vert^2  ( i \frac18 c\nab^{-1})  e^{-i s c\nab} \overline{\mathcal{I}(s,c\nab, u(0))} \right)ds\\
&+  \mathcal{B}_1(t,u)
\end{aligned}
\end{equation}
where we have used the notation \eqref{osc}, i.e., that 
$$
\mathcal{I}(s,c\nab, u(0))=\int_0^s   e^{-i s_1 c\nab}\left( e^{i s_1 c\nab} u(0) + e^{-i s_1 c\nab}\overline{u}(0) \right)^3 ds_1.
$$
The remainder $ \mathcal{B}_1(t,u)$ thereby satisfies
\begin{equation}\label{B1}
\begin{aligned}
\Vert  \mathcal{B}_1(t,u)\Vert_r &  \leq t^3 K(\sup_{0\leq t\leq s} \Vert u(t)\Vert_r).
\end{aligned}
\end{equation}
Thanks to the first order scheme we know that
\begin{align*}
\mathcal{I}(s,c\nab, u(0)) & =  { s}\varphi_1(2i c^2s) (u(0))^3 +  3 { s}(u(0))^2 \varphi_1(-2is  \Lc) \overline{u(0)}
\\ & +  3 { s}     u(0) \varphi_1({- 2 i s c\nab })(\overline{u(0)})^2+    { s}\varphi_1({- 2 i s ( c\nab+c^2) })(\overline{u(0)})^3\\
&  + \mathcal{B}_2(s,u)
\end{align*}
 where the remainder satisfies
\begin{equation*}\label{B2}
\begin{aligned}
\Vert  \mathcal{B}_2(s,u)\Vert_r &  \leq s^2 K(\sup_{0\leq t\leq s} \Vert u(t)\Vert_{r+1}).
\end{aligned}
\end{equation*}
Together with the expansion $c\nab = c^2 + \mathcal{O}(\Delta)$ (cf. \eqref{bound:Lc}) which formally implies that
$$
 s \varphi_1({- \ell i s c\nab }) = \int_0^s e^{- \ell i s_1 c\nab}ds_1 = 
 \int_0^s\left( e^{- \ell i s_1 c^2 }+ \mathcal{O}(s_1 \Delta)\right) ds_1=
  s \varphi_1({- \ell i s c^2 })  + \mathcal{O}(s^2 \Delta)
$$
we thus obtain
\begin{equation}\label{Duh3a}
\begin{aligned}
u(t)  
& = e^{i t c\nab}u(0)- i \frac18 c\nab^{-1}  e^{i t c\nab}\mathcal{I}(t,c\nab, u(0))\\
&- 3i \frac18 c\nab^{-1}  e^{i t c^2}\int_0^t   e^{-i s c^2}
\left(\Big(e^{i s c^2}u(0)\Big)^2 \widetilde{\mathcal{I}}(s,c^2, u(0)) \right)ds\\
&- 3i \frac18 c\nab^{-1}  e^{i t c^2}\int_0^t   e^{-i s c^2}
\left(\Big(e^{i s c^2}u(0)\Big)^2 \overline{\widetilde{\mathcal{I}}(s,c^2, u(0))} \right)ds\\
&- 3i \frac18 c\nab^{-1}  e^{i t c^2}\int_0^t   e^{-i s c^2}
\left(\Big(e^{-i s c^2}\overline{u(0)}\Big)^2 \widetilde{\mathcal{I}}(s,c^2, u(0)) \right)ds\\
&- 3i \frac18 c\nab^{-1}  e^{i t c^2}\int_0^t   e^{-i s c^2}
\left(\Big(e^{-i s c^2}\overline{u(0)}\Big)^2 \overline{\widetilde{\mathcal{I}}(s,c^2, u(0))} \right)ds\\
&-6 i \frac18 c\nab^{-1}  e^{i t c^2}\int_0^t   e^{-i s c^2}
\left(\Big\vert e^{i s c^2}u(0)\Big\vert^2 \widetilde{\mathcal{I}}(s,c^2, u(0)) \right)ds \\
&-6 i \frac18 c\nab^{-1}  e^{i t c^2}\int_0^t   e^{-i s c^2}
\left(\Big\vert e^{i s c^2}u(0)\Big\vert^2 \overline{\widetilde{\mathcal{I}}(s,c^2, u(0))} \right)ds \\
&+   \mathcal{B}_3(t,u)
\end{aligned}
\end{equation}
 with
\begin{align*}
\widetilde{\mathcal{I}}(s,c\nab, u(0)) & = - i  s \frac18  c\nab^{-1}  e^{is  c^2}
\Big[( \varphi_1(2i c^2s) (u(0))^3 +  3 \vert u(0)\vert^2 {u(0)}
\\ & +  3      u(0) \varphi_1({- 2 i s c^2 })(\overline{u(0)})^2+    \varphi_1({- 4 i s c^2 })(\overline{u(0)})^3 \Big]
\end{align*}
 and where the remainder $\mathcal{B}_3(t,u)$ satisfies
\begin{equation}\label{B3}
\begin{aligned}
\Vert  \mathcal{B}_3(t,u)\Vert_r &  \leq t^3 K(\sup_{0\leq t\leq s} \Vert u(t)\Vert_{r+2}).
\end{aligned}
\end{equation}
Next we state an essential Lemma on the integration of the oscillations.
\begin{lemma}
For $m,l\in\mathbb{Z}$, $l\neq 0$, $c\geq0$ and $t\geq0$ we have
$$
\int_0^t s e^{imc^2s}\varphi_1(i\ell c^2s)ds = \frac{t}{i\ell c^2}\big( \varphi_1(i(\ell+m)c^2t) - \varphi_1(i{m }c^2t) \big).
$$
\end{lemma}
\begin{proof}
The assertion follows thanks to 
\begin{align*}
    \int_0^t s e^{imc^2s}\varphi_1(ilc^2s)ds = \int_0^t s \frac{e^{i(m+\ell)c^2s}-e^{i { m }c^2s}}{i\ell c^2s} = \frac{t}{i\ell c^2}\big( \varphi_1(i(\ell+m)c^2t) - \varphi_1(i{m }c^2t) \big).
\end{align*}
\end{proof}
In particular, we have
\begin{equation}\label{def:phi2}
\int_0^t s \varphi_1(i \ell c^2s) ds = \frac{1}{i \ell c^2} \int_0^t \left( e^{i \ell c^2 s} - 1 \right) ds =
\frac{t}{i \ell c^2} \left(\varphi_1({i \ell c^2 t}) - 1\right) =:t^2 \varphi_2({i \ell c^2 t}) .
\end{equation} 
Furthermore, for $m\in \mathbb{Z}$, we define
\begin{align}
\label{def_Psi_2}
   &   t^2 \Psi_2(imc^2t) := \int_0^t s e^{imc^2s} ds.
\end{align}

With this definition at hand we obtain that
\begin{equation}\label{M1}
\begin{aligned}
- i \frac18 c\nab^{-1}  e^{i t c^2} &\int_0^t   e^{-i s c^2}
\left(\Big(e^{i s c^2}u(0)\Big)^2 \widetilde{\mathcal{I}}(s,c^2, u(0)) \right)ds
\\& = - i \frac18 c\nab^{-1}  e^{i t c^2} \int_0^t e^{2i s c^2} u(0)^2
\Big(- i  s \frac18  c\nab^{-1} 
\Big[ \varphi_1(2i c^2s) (u(0))^3 +  3 \vert u(0)\vert^2 u(0)
\\ & +  3      {u(0)}\varphi_1({-2 i s c^2 })( \overline{u(0)})^2+    \varphi_1({-4 i s c^2 })(\overline{u(0)})^3 \Big]\Big)
ds
\\& = -  \frac18 \cdot\frac18 c\nab^{-1}  e^{i t c^2} \Big( u(0)^2
  c\nab^{-1} 
\Big[ \frac{t}{2ic^2}(\varphi_1(4i c^2t)- \varphi_1(2i c^2t)) (u(0))^3 
\\&+  3t^2 {\Psi_2}(2ic^2t)\vert u(0)\vert^2 u(0)+  3  t^2    {u(0)}\varphi_2({2 i t c^2 })( \overline{u(0)})^2 
\\&+    \frac{t}{-4ic^2}( \varphi_1(-2i c^2t)- \varphi_1(2i c^2t) )(\overline{u(0)})^3 \Big]
\Big) =: \mathcal{M}_1(t,c^2,u(0))
\end{aligned}
\end{equation}
\begin{equation}\label{M2}
\begin{aligned}
    - i \frac18 c\nab^{-1}  e^{i t c^2} &\int_0^t   e^{-i s c^2}
\left(\Big(e^{i s c^2}u(0)\Big)^2 \overline{\widetilde{\mathcal{I}}(s,c^2, u(0))} \right)ds
\\&=- i \frac18 c\nab^{-1}  e^{i t c^2} \int_0^t  u(0)^2
\Big( i  s \frac18  c\nab^{-1} 
\Big[ \varphi_1(-2i c^2s) (\overline{u(0)})^3 +  3 \vert u(0)\vert^2 \overline{u(0)}
\\ & +  3      \overline{u(0)}\varphi_1({2 i s c^2 })( {u(0)})^2+    \varphi_1({4 i s c^2 })({u(0)})^3 \Big]\Big)
ds
\\& = + t^2 \frac{1}{64} c\nab^{-1}e^{itc^2}\Big( u(0)^2 c\nab^{-1} \Big[ \varphi_2(- 2ic^2t)\overline{u(0)}^3 + \frac32 \vert u(0)\vert^2\overline{u(0)} 
\\&+3\overline{u(0)}\varphi_2(2itc^2){u(0)}^2 + \varphi_2(4itc^2){u(0)}^3  \Big]\Big)  =: \mathcal{M}_2(t,c^2,u(0))
\end{aligned}
\end{equation}
\begin{equation}\label{M3}
\begin{aligned}
    - i \frac18 c\nab^{-1}  e^{i t c^2}&\int_0^t   e^{-i s c^2}
\left(\Big(e^{-i s c^2}\overline{u(0)}\Big)^2 \widetilde{\mathcal{I}}(s,c^2, u(0)) \right)ds
\\&  = - i \frac18 c\nab^{-1}  e^{i t c^2}\int_0^t   e^{-2i s c^2}  \overline{u(0)}^2 \Big(- i  s \frac18  c\nab^{-1} 
\Big[ \varphi_1(2i c^2s) (u(0))^3 +  3 \vert u(0)\vert^2 u(0)
\\ & +  3      {u(0)}\varphi_1({-2 i s c^2 }) \overline{u(0)}^2+    \varphi_1({-4 i s c^2 })(\overline{u(0)})^3 \Big]\Big)
\\&= - \frac{1}{64} c\nab^{-1}  e^{i t c^2} \Big(  \overline{u(0)}^2 c\nab^{-1} \Big[ t^2{\varphi_2}(-2ic^2t)u(0)^3 + 3 t^2 {\Psi_2 }(-2ic^2t)\vert u(0)\vert^2 u(0) 
\\ &+ \frac{3t}{-2ic^2}u(0)(\varphi_1(-4ic^2t) - \varphi_1(-2ic^2t))\overline{u(0)}^2 +\frac{t}{-4ic^2}(\varphi_1(-6itc^2)-\varphi_1(-2itc^2))\overline{u(0)}^3\Big] \Big)
 \\&=: \mathcal{M}_3(t,c^2,u(0))
\end{aligned}
\end{equation}
\begin{equation}\label{M4}
\begin{aligned}
    - i \frac18 c\nab^{-1}  e^{i t c^2}&\int_0^t   e^{-i s c^2}
\left(\Big(e^{-i s c^2}\overline{u(0)}\Big)^2 \overline{\widetilde{\mathcal{I}}(s,c^2, u(0))} \right)ds
\\&  = - i \frac18 c\nab^{-1}  e^{i t c^2}\int_0^t   e^{-4i s c^2}  \overline{u(0)}^2 \Big( i  s \frac18  c\nab^{-1} 
\Big[ \varphi_1(-2i c^2s) (\overline{u(0)})^3 +  3 \vert u(0)\vert^2 \overline{u(0)}
\\ & +  3      \overline{u(0)}\varphi_1({2 i s c^2 })( {u(0)})^2+    \varphi_1({4 i s c^2 })({u(0)})^3 \Big]\Big)
\\&=  \frac{1}{64} c\nab^{-1}  e^{i t c^2} \Big(  \overline{u(0)}^2 c\nab^{-1} \Big[ \frac{t}{-2ic^2}(\varphi_1(-6ic^2t)-\varphi_1(-4ic^2t))\overline{u(0)}^3 
\\& + 3t^2{\Psi_2}(-4ic^2t)\vert u(0)\vert^2\overline{u(0)} +\frac{3t}{2ic^2}\overline{u(0)}(\varphi_1(-2ic^2t)-\varphi_1(-4ic^2t))u(0)^2
\\& + t^2{\varphi_2} (-4ic^2t)u(0)^3\Big] \Big)  =: \mathcal{M}_4(t,c^2,u(0))
\end{aligned}
\end{equation}
\begin{equation}\label{M5}
\begin{aligned}
    - i \frac18 c\nab^{-1}  e^{i t c^2}&\int_0^t   e^{-i s c^2}
\left(\Big\vert e^{i s c^2}u(0)\Big\vert^2 \widetilde{\mathcal{I}}(s,c^2, u(0)) \right)ds
\\& = - i \frac18 c\nab^{-1}  e^{i t c^2} \int_0^t  \vert u(0)\vert^2
\Big(- i  s \frac18  c\nab^{-1} 
\Big[ \varphi_1(2i c^2s) (u(0))^3 +  3 \vert u(0)\vert^2 u(0)
\\ & +  3      {u(0)}\varphi_1({-2 i s c^2 })( \overline{u(0)})^2+    \varphi_1({-4 i s c^2 })(\overline{u(0)})^3 \Big]\Big)ds
\\& = -\frac{t^2}{64} c\nab^{-1}  e^{i t c^2} \Big( \vert u(0)\vert^2 c\nab^{-1} \Big[ \varphi_2(2ic^2t)u(0)^3 + \frac32\vert u(0)\vert^2u(0) 
\\&+ 3u(0)\varphi_2(-2ic^2t)\overline{u(0)}^2 + \varphi_2(-4ic^2t)\overline{u(0)}^3  \Big] \Big)  =: \mathcal{M}_5(t,c^2,u(0))
\end{aligned}
\end{equation}
\begin{equation}\label{M6}
\begin{aligned}
    - i \frac18 c\nab^{-1}  e^{i t c^2}&\int_0^t   e^{-i s c^2}
\left(\Big\vert e^{i s c^2}u(0)\Big\vert^2 \overline{\widetilde{\mathcal{I}}(s,c^2, u(0))} \right)ds
\\& = - i \frac18 c\nab^{-1}  e^{i t c^2} \int_0^t e^{-2isc^2} \vert u(0)\vert^2
\Big( i  s \frac18  c\nab^{-1} 
\Big[ \varphi_1(-2i c^2s) (\overline{u(0)})^3 +  3 \vert u(0)\vert^2 \overline{u(0)}
\\ & +  3      \overline{u(0)}\varphi_1({2 i s c^2 })( {u(0)})^2+    \varphi_1({4 i s c^2 })({u(0)})^3 \Big]\Big)
ds
\\& = \frac{1}{64} c\nab^{-1}  e^{i t c^2} \Big( \vert u(0)\vert^2 c\nab^{-1} \Big[ \frac{t}{-2ic^2}(\varphi_1(-4ic^2t)-\varphi_1(-2ic^2t))\overline{u(0)}^3 
\\&+ 3u(0)t^2{\Psi_2}(-2ic^2t) \overline{u(0)}^2 +3\overline{u(0)} t^2{\varphi_2}(-2ic^2t)u(0)^2 
\\&+ \frac{t}{4ic^2}(\varphi_1(2ic^2t) - \varphi_1(-2ic^2t)){u(0)}^3 \Big] \Big)  =: \mathcal{M}_6(t,c^2,u(0)).
\end{aligned}
\end{equation}
The above calculations allow us to express \eqref{Duh3a} as follows
\begin{equation}\label{Duh3}
\begin{aligned}
u(t)  
& = e^{i t c\nab}u(0)- i \frac18 c\nab^{-1}  e^{i t c\nab}\mathcal{I}(t,c\nab, u(0))\\
&+ 3 \sum_{j=1}^4 \mathcal{M}_j(t,c^2,u(0)) + 6 \mathcal{M}_5(t,c^2,u(0)) + 6 \mathcal{M}_6(t,c^2,u(0))
\\&+   \mathcal{B}_3(t,u),
\end{aligned}
\end{equation}
where the remainder satisfies \eqref{B3}. Thus, it remains to derive a second order uniformly accurate, low regularity approximation to the central oscillations $\mathcal{I}(t,c\nab, u(0))$ defined in \eqref{osc}. We will again treat each integral separtely 
\begin{align}\label{I12}
\mathcal{I}(t, c\nab, u(0)) &=\int_0^t   e^{-i s (c^2+\Lc)}\left(e^{i s (c^2+\Lc)}u(0)\right)^3 ds\\\label{I22}
& + 3  \int_0^t   e^{-i s  (c^2+\Lc)}\left[\left(e^{i s  (c^2+\Lc)}u(0)\right)^2  e^{-i s  (c^2+\Lc)} \overline{u}(0)  \right]ds\\\label{I32}
& + 3  \int_0^t   e^{-i s  (c^2+\Lc)}\left[\left(e^{-i s  (c^2+\Lc)} \overline{u}(0)\right)^2  e^{i s  (c^2+\Lc)}  {u}(0)  \right]ds\\\label{I42}
& +\int_0^t   e^{-i s  (c^2+\Lc)}\left(e^{-i s (c^2+\Lc)}\overline u(0)\right)^3 ds.
\end{align}
\begin{lemma}[Second order approximation of the integral \eqref{I12}]\label{lem:I12}
It holds that
\begin{align*}
\int_0^t  &  e^{-i s (c^2+\Lc)}\left(e^{i s (c^2+\Lc)}v\right)^3 ds \\ & = t \varphi_1(2i c^2t) v^3 + t\varphi_2 (2i c^2t)    \left(  e^{-i t \Lc}
\left(e^{i t  \Lc}v\right)^3 - v^3\right)
+ \mathcal{O}\left( t^3 \mathcal{C}^2[f_{\text{cub}}(\cdot,\cdot,\cdot), \Lc](v,v,v)  \right).
\end{align*}
\end{lemma}
\begin{proof}
As in Lemma \ref{lem:I1} we define the filtered function
$$
\mathcal{N}(s_1,\Lc,v) =  e^{-i s_1 \Lc}
\left(e^{i s_1  \Lc}v\right)^3
$$
which allows us to express the oscillatory integral \eqref{I12} as follows
\begin{align*}
\mathcal{I}_1(t, c\nab, v) 
& = \int_0^t e^{2i c^2s}   \mathcal{N}(s,\Lc,v)  ds.
\end{align*}
Now we employ a second order expansion of $\mathcal{N}(s,\Lc,v) $ around $s =0$
\[
\mathcal{N}(s,\Lc,v) = \mathcal{N}(0,\Lc,v) +s  \partial_{s} \mathcal{N}(0,\Lc,v) + \int_0^{s}\int_0^{s_1}  \partial_{\xi }^2  \mathcal{N}(\xi,\Lc,v)
d\xi ds_1
\]
where  $ \partial_{s_1}^2  \mathcal{N}(s_1,\Lc,v) $ obeys the improved error structure
\begin{align}\label{2com}
\partial_{s_1}^2\mathcal{N}(s_1,\Lc,v) 
= e^{-i s_1 \Lc}  \mathcal{C}^2\left[f,  { \Lc} \right]
 \left(e^{i s_1  \Lc}v,e^{i s_1  \Lc}v,e^{i s_1  \Lc}v \right)
\end{align}
where we recall that 
$$ \mathcal{C}^2\left[f,  { \Lc}\right] = \mathcal{C}\left[
 \mathcal{C}\left[
f, { \Lc}
\right], { \Lc}
\right].$$
{
Thanks to Lemma \ref{bound:comm} and the fact that $e^{i s_1 \Lc}$ is a linear isometry in $H^r$ we have that
\begin{align*}
\left \Vert 
 \int_0^{s}\int_0^{s_1}    \partial_{\xi }^2  \mathcal{N}(\xi,\Lc,v)  d\xi ds_1 \right \Vert_r
& \leq 
 \int_0^{s}\int_0^{s_1}    \left \Vert \mathcal{C}^2\left[f, \mathcal{L}_c\right]
 \left(e^{i s_1  \Lc}v,e^{i s_1  \Lc}v,e^{i s_1  \Lc}v \right) \right \Vert_r 
  d\xi ds_1 \\
  & \leq k 
 \int_0^{s}\int_0^{s_1}    \left \Vert  e^{i s_1  \Lc}v  \right \Vert_{r+2}^3
  d\xi ds_1 \\& 
   \leq k s^2  \left\Vert  
 v
 \right \Vert_{r+2}^3.
\end{align*}
}
Thus,  we obtain that
\begin{align}\label{I112}
\mathcal{I}_1(t, c\nab, v) 
& = \int_0^t e^{2i c^2s}   \left( v^3 +s  \partial_{s} \mathcal{N}(0,\Lc,v) \right)  ds + \mathcal{O}\left( t^3 \mathcal{C}^2[f_{\text{cub}}(\cdot,\cdot,\cdot), \Lc](v,v,v)  \right).
\end{align}
where we have used that $ \mathcal{N}(0,\Lc,v) = v^3$.

 In order to guarantee stability of the scheme we will not explicitly embed $ \partial_{s} \mathcal{N}(0,\Lc,v)   $ into our scheme. Instead we will exploit that  formally we have for some $0 \leq\eta\leq t$ that
\begin{equation*}
\partial_s \mathcal{N}(0,\Lc,v)  =  \frac{\delta_{s } \mathcal{N}(s,\Lc,v) }{t}\vert_{s = t}+ \mathcal{O}\left(t \partial_{s}^2 \mathcal{N}(0,\Lc,v)\right) 
\end{equation*}
with the standard shift operator $$\delta_\zeta g(\zeta) := g(\zeta)- g(0)$$
such that
$$
\frac{\delta_{\zeta } \mathcal{N}(\zeta,\Lc,v) }{t}\vert_{\zeta = t} = \frac{1}{t} \left(  e^{-i t \Lc}
\left(e^{i t  \Lc}v\right)^3 - v^3\right) .
$$
Together with \eqref{I112} and the definition of the $\varphi_2$ function (see \eqref{def:phi2}) this yields the assertion.
\end{proof}
\begin{lemma}[Second order approximation of the integral \eqref{I22}]\label{lem:I22}
It holds that
\begin{align*}
3  \int_0^t  &e^{-i s  (c^2+\Lc)}\left[\left(e^{i s  (c^2+\Lc)}v\right)^2  e^{-i s  (c^2+\Lc)} \overline{v}\right]ds \\& =  3t v^2  \varphi_1(-2i t \Lc) \overline v+
3t 
\left(e^{-i t   \Lc}\left[\left(e^{i t   \Lc}v \right)^2  e^{ i t   \Lc}  \varphi_2(-2i t \Lc) \overline{v}  \right] 
  - v^2  \varphi_2(-2i t \Lc) \overline{v}
\right)\\& + \mathcal{O}\left( t^3 \mathcal{C}^2[f_{\text{cub}}(\cdot,\cdot,\cdot), \Lc](v,v,\overline v)  \right).
\end{align*}
\end{lemma}
\begin{proof}
The proof follows similarly to the proof of Lemma \ref{lem:I12} by including the approximation
\begin{equation*}
\partial_s \mathcal{N}(s, 0,\Lc,v)  =  \frac{\delta_{s_1} \mathcal{N}(s,s_1 ,\Lc,v) }{t}\vert_{s_1= t}+ \mathcal{O}\left(t \partial_{s_1}^2 \mathcal{N}(s,s_1,\Lc,v)\right) 
\end{equation*}
for the appropriate filtered function (cf. \eqref{filter2})
\begin{equation*}
\mathcal{N}(s, s_1,\Lc,v) =  e^{-i s_1   \Lc}\left[\left(e^{i s_1   \Lc}v \right)^2  e^{ i s_1   \Lc} e^{-2 i s   \Lc} \overline{v}  \right] 
\end{equation*}
into the numerical discretisation. With the observation that
\begin{equation*}
\partial_s \mathcal{N}(s, 0,\Lc,v)  =  \frac{1}{t}
\left(
  e^{-i t   \Lc}\left[\left(e^{i t   \Lc}v \right)^2  e^{ i t   \Lc} e^{-2 i s   \Lc} \overline{v}  \right] 
  - v^2 e^{-2 i s   \Lc} \overline{v}
\right)+ \mathcal{O}\left(t \partial_{s_1}^2 \mathcal{N}(s,s_1,\Lc,v)\right) 
\end{equation*}
we thus obtain
\begin{align*}
  \int_0^t  & e^{-i s  (c^2+\Lc)}\left[\left(e^{i s  (c^2+\Lc)}v\right)^2  e^{-i s  (c^2+\Lc)} \overline{v}\right]ds
\\ &= \int_0^t  \left( \mathcal{N}(s, 0,\Lc,v) + s  \partial_{s_1}  \mathcal{N}(s, 0,\Lc,v) + \int_0^{s}
\partial_{s_1}^2 \mathcal{N}(s, s_1 ,\Lc,v) d\xi\right) ds\\
& =  t v^2  \varphi_1(-2i t \Lc) \overline v+
t 
\left(
  e^{-i t   \Lc}\left[\left(e^{i t   \Lc}v \right)^2  e^{ i t   \Lc}  \varphi_2(-2i t \Lc) \overline{v}  \right] 
  - v^2  \varphi_2(-2i t \Lc) \overline{v}
\right)\\& + \mathcal{O}\left( t^3 \mathcal{C}^2[f_{\text{cub}}(\cdot,\cdot,\cdot), \Lc](v,v,\overline v)  \right).
\end{align*}
\end{proof}
\begin{lemma}[Approximation of the integral \eqref{I32}]\label{lem:I32}
It holds that
\begin{align*}
 3  \int_0^t &   e^{-i s  (c^2+\Lc)}\left[\left(e^{-i s  (c^2+\Lc)} \overline{v}\right)^2  e^{i s  (c^2+\Lc)}  v \right]ds  \\ & = 3  t     v \varphi_1({- 2 i t c\nab })\overline{v}^2 
 + 3 t \Big( e^{-i t   \Lc}\left[\left(  e^{ i t   \Lc}  {v} \right) \varphi_2({- 2 i t c\nab })    e^{ i t   \Lc}  \overline v^2 \right]
 - v \varphi_2({- 2 i t c\nab })   \overline v^2\Big)\\
 & + 3t \Big( v  \varphi_2({- 2 i t c\nab })   e^{ it   \Lc} \left(e^{-i t   \Lc}\overline v \right)^2 - v\varphi_2({- 2 i t c\nab })     \overline v^2 \Big)
 +  \mathcal{O}\left( t^2 \mathcal{C}[f_{\text{cub}}(\cdot,\cdot,\cdot), \Lc](v,\overline v,\overline v)  \right).
\end{align*}
\end{lemma}
\begin{proof}
Again we define the three scale expansion
$$
\mathcal{N}(s, s_1,s_2, \Lc,v) =  e^{-i s_1   \Lc}\left[\left(  e^{ i s_1   \Lc}  {v} \right)e^{- 2i s   \Lc}   e^{ i s_1   \Lc}e^{ i s_2   \Lc} \left(e^{-i s_2   \Lc}\overline v \right)^2 \right]
$$ 
such that (cf. proof of Lemma \ref{lem:I3})
\begin{align*}  \int_0^t   & e^{-i s  (c^2+\Lc)}\left[\left(e^{-i s  (c^2+\Lc)} \overline{v}\right)^2  e^{i s  (c^2+\Lc)}  v \right]ds = \int_0^t e^{-2 i s  c^2} \mathcal{N}(s, s,s, \Lc,v)ds \\
& = \int_0^t v  e^{- 2i s   (\Lc+c^2)} \overline v^2  ds +\int_0^t  s \Big( \partial_{s_1} \mathcal{N}(s, 0  ,s , \Lc,v) + \partial_{s_2} \mathcal{N}(s, 0  ,0, \Lc,v)  d s_2  \Big) ds\\
& +  \mathcal{O}\left(\partial_{s_1}^2 \mathcal{N}\right)+  \mathcal{O}\left(\partial_{s_2}^2 \mathcal{N} \right)\\
& = \int_0^t v  e^{- 2i s   (\Lc+c^2)} \overline v^2  ds +\int_0^t  s \Big( \partial_{s_1} \mathcal{N}(s, 0  ,0 , \Lc,v) + \partial_{s_2} \mathcal{N}(s, 0  ,0, \Lc,v)  d s_2  \Big) ds\\
& +  \mathcal{O}\left(t^3 \partial_{s_1}^2 \mathcal{N} \right)+  \mathcal{O}\left(t^3 \partial_{s_2}^2 \mathcal{N} \right) +\mathcal{O}\left(t^3 \partial_{s_1} \partial_{s_2} \mathcal{N} \right).
  \end{align*}
  Next we use that
\begin{align*}
&  \partial_{s_1} \mathcal{N}(s, 0  ,0 , \Lc,v)  = \frac{1}{t} \left( \mathcal{N}(s, t  ,0 , \Lc,v)  -  \mathcal{N}(s, 0  ,0 , \Lc,v) \right)+  \mathcal{O}\left(t \partial_{s_1}^2 \mathcal{N} \right) \\
   & \partial_{s_2} \mathcal{N}(s, 0  ,0 , \Lc,v)  = \frac{1}{t} \left( \mathcal{N}(s, 0  ,t , \Lc,v)  -  \mathcal{N}(s, 0  ,0 , \Lc,v) \right)+  \mathcal{O}\left(t \partial_{s_2}^2 \mathcal{N} \right)
  \end{align*}
  as well as that
 \begin{align*}
&  \mathcal{N}(s, t  ,0 , \Lc,v)  -  \mathcal{N}(s, 0  ,0 , \Lc,v)  = 
 e^{-i t   \Lc}\left[\left(  e^{ i t   \Lc}  {v} \right)e^{- 2i s   \Lc}   e^{ i t   \Lc}  \overline v^2 \right]
 - ve^{- 2i s   \Lc}   \overline v^2\\
 & \mathcal{N}(s, 0  ,t , \Lc,v)  -  \mathcal{N}(s, 0  ,0 , \Lc,v) =
v  e^{- 2i s   \Lc}   e^{ it   \Lc} \left(e^{-i t   \Lc}\overline v \right)^2 - ve^{- 2i s   \Lc}   \overline v^2.
  \end{align*}
  This implies the assertion.
\end{proof}
\begin{lemma}[Approximation of the integral \eqref{I42}]\label{lem:I42}
It holds that
\begin{align*}
 \int_0^t  &  e^{-i s  (c^2+\Lc)}\left(e^{-i s (c^2+\Lc)}\overline v \right)^3 ds  \\&
  =   t      \varphi_1({- 2 i t ( c\nab+c^2) })\overline{v}^3  
  + t \left(e^{i t  \Lc}  \varphi_2({- 2 i t ( c\nab+c^2) })  \left(e^{-i t   \Lc}\overline v \right)^3 -  \varphi_2({- 2 i t ( c\nab+c^2) }) \overline v^3\right)
  \\&+  \mathcal{O}\left( t^3 \mathcal{C}^2[f_{\text{cub}}(\cdot,\cdot,\cdot), \Lc](\overline v,\overline v,\overline v)  \right).
\end{align*}
\end{lemma}
\begin{proof}
The assertion follows similarly to the previous lemmata  by choosing the appropriate filtered function
$$
\mathcal{N}(s, s_1, \Lc,v) =  e^{i s_1   \Lc}  e^{- 2i s   \Lc}   \left(e^{-i s_1   \Lc}\overline v \right)^3 .
$$

\end{proof}
Plugging the above lemmata into \eqref{Duh3} yields together with the double commutator bound of Lemma \ref{bound:comm} that for $v= u(0)$
\begin{equation}\label{Duh4}
\begin{aligned}
u(t)  
& = e^{i t c\nab}v- i \frac18 c\nab^{-1}  e^{i t c\nab}\Big\{t \varphi_1(2i c^2t) v^3 + t\varphi_2 (2i c^2t)    \left(  e^{-i t \Lc}
\left(e^{i t  \Lc}v\right)^3 - v^3\right)
\\
&
+3t v^2  \varphi_1(-2i t \Lc) \overline v+
3t 
\left(e^{-i t   \Lc}\left[\left(e^{i t   \Lc}v \right)^2  e^{ i t   \Lc}  \varphi_2(-2i t \Lc) \overline{v}  \right] 
  - v^2  \varphi_2(-2i t \Lc) \overline{v}
\right)\\
& +3  t     v \varphi_1({- 2 i t c\nab })\overline{v}^2 
 + 3 t \Big( e^{-i t   \Lc}\left[\left(  e^{ i t   \Lc}  {v} \right) \varphi_2({- 2 i t c\nab })    e^{ i t   \Lc}  \overline v^2 \right]
 - v \varphi_2({- 2 i t c\nab })   \overline v^2\Big)\\
 & + 3t \Big( v  \varphi_2({- 2 i t c\nab })   e^{ it   \Lc} \left(e^{-i t   \Lc}\overline v \right)^2 - v\varphi_2({- 2 i t c\nab })     \overline v^2 \Big)\\&
 +  t      \varphi_1({- 2 i t ( c\nab+c^2) })\overline{v}^3  
  + t \left(e^{i t  \Lc}  \varphi_2({- 2 i t ( c\nab+c^2) })  \left(e^{-i t   \Lc}\overline v \right)^3 -  \varphi_2({- 2 i t ( c\nab+c^2) })\overline v^3\right)
\Big\}\\
&+  3 \sum_{j=1}^4 \mathcal{M}_j(t,c^2,v) + 6 \mathcal{M}_5(t,c^2,v) + 6 \mathcal{M}_6(t,c^2,v)+   \mathcal{B}_3(t,u),
\end{aligned}
\end{equation}
where the remainder $ \mathcal{B}_3(t,u)$ satisfies \eqref{B3}. The expansion \eqref{Duh4} motivates the second order uniformly accurate low regularity scheme
\begin{equation}\label{scheme2}
\begin{aligned}
u^{n+1} 
& = e^{i \tau c\nab}u^n- i \tau \frac18 c\nab^{-1}  e^{i \tau c\nab}\Big\{  \varphi_1(2i c^2\tau) (u^n)^3 +  \varphi_2 (2i c^2\tau)    \left(  e^{-i \tau \Lc}
\left(e^{i \tau  \Lc}u^n\right)^3 - (u^n)^3\right)
\\
&
+3 (u^n)^2  \varphi_1(-2i \tau \Lc) \overline{u^n}+
3
\left(e^{-i \tau   \Lc}\left[\left(e^{i \tau   \Lc}u^n \right)^2  e^{ i \tau   \Lc}  \varphi_2(-2i \tau \Lc) \overline{u^n}  \right] 
  - (u^n)^2  \varphi_2(-2i \tau \Lc) \overline{u^n}
\right)\\
& +3      u^n \varphi_1({- 2 i \tau c\nab }) \overline{(u^n)^2}
 + 3   \Big( e^{-i \tau   \Lc}\left[\left(  e^{ i \tau   \Lc}  {u^n} \right) \varphi_2({- 2 i \tau c\nab })    e^{ i \tau   \Lc}   \overline{(u^n)^2} \right]
 - u^n \varphi_2({- 2 i \tau c\nab })   \overline{(u^n)^2}\Big)\\
 & + 3 \Big( u^n  \varphi_2({- 2 i \tau c\nab })   e^{ i\tau   \Lc} \left(e^{-i \tau   \Lc}\overline{u^n} \right)^2 - u^n \varphi_2({- 2 i \tau c\nab })     \overline{(u^n)^2}  \Big)\\&
 +        \varphi_1({- 2 i \tau ( c\nab+c^2) }) \overline{(u^n)^3} 
  + \left(e^{i \tau  \Lc}  \varphi_2( { - 2 i \tau ( c\nab+c^2) })  \left(e^{-i \tau   \Lc}\overline{u^n} \right)^3 -  \varphi_2({ - 2 i \tau ( c\nab+c^2) })\overline{(u^n)^3} \right)
\Big\}\\
&+  3 \sum_{j=1}^4 \mathcal{M}_j(\tau,c^2,u^n) + 6 \mathcal{M}_5(\tau,c^2,u^n) + 6 \mathcal{M}_6(\tau,c^2,u^n)
\end{aligned}
\end{equation}
where $\mathcal{M}_{1,2,3,4,5,6}(\tau,c^2,u^n)$ are defined in \eqref{M1} up to \eqref{M6}. The following global error estimates holds true for the second order scheme \eqref{scheme2}.
\begin{theorem} \label{thm:glob2} Fix  $r>d/2$ and assume that the solution of \eqref{eq:kgr} satisfies $u \in \mathcal{C}([0,T];H^{r+2})$.  Then there exists a $\tau_0>0$ such that for all $0<\tau \leq \tau_0$  the following global error estimate holds for  $u^n$ defined in~\eqref{scheme2}
\begin{align*}
\Vert u(t_n)- u^n \Vert_{r} &  \leq    \tau^2  K\left(\sup_{0\leq t\leq t_{n}} \Vert u(t) \Vert_{r+2}\right) ,
\end{align*}
where $K$ depends on $t_n$ and the $H^{r+2}$ norm of the solution $u$, but can be chosen independently of $c$.
\end{theorem}
\begin{proof}
The proof follows the line of argumentation as the proof of Theorem \ref{thm:glob} where we note that the local error $ \mathcal{B}_3(t,u)$ (compare \eqref{Duh4} and \eqref{scheme2}) satisfies \eqref{B3} and only involves second order derivatives of the solution.
\end{proof}
{ 
\begin{rem}
Note that by  \eqref{eq:BCc} we have that $u(0) \in H^{r+2}$ if
\begin{equation}\label{condz2}
z(0) \in H^{r+2}, \quad c^{-1}\nab^{-1} z'(0)\in H^{r+2}.
\end{equation}
Hence, if the solution $z$ of \eqref{eq:kgr} satisfies \eqref{condz} we can conclude by  local wellposedness of \eqref{eq:NLSc} that there exists a $T=T_{r+2}>0$ such that $u \in \mathcal{C}([0,T_{r+2}];H^{r+2})$.
\end{rem}}
\begin{rem}\label{thm:asymp2}
Again we can show that the scheme \eqref{scheme2} asymptotically $($in the limit $c\to \infty)$ converges to a second order low regularity integrator for   NLS \eqref{NLSlimit}  in the sense that for sufficiently smooth solutions
$$
\Vert u^n - e^{ic^2 t_n} u_{\ast,\infty}^n\Vert_r \leq k c^{-1}
$$
where $u_{\ast,\infty}^n$ is a second order low regularity approximation of the NLS  limit equation \eqref{NLSlimit} $($see \cite{BS20}$)$
\begin{equation}\label{scheme2NLS}
\begin{aligned}
u_{\ast,\infty}^{n+1} 
& = e^{-i \tau \frac{1}{2}\Delta }u_{\ast,\infty}^{n}- i \tau \frac38   e^{-i \tau  \frac{1}{2}\Delta}\Big[ 
\left( u_{\ast,\infty}^{n}\right)^2 \big(\varphi_1(i \tau \Delta) - \varphi_2(i \tau \Delta) \big) \overline{u_{\ast,\infty}^{n}}\\& 
\quad + \left(e^{-i \tau \frac{1}{2}\Delta }u_{\ast,\infty}^{n}\right)^2 \varphi_2(i\tau\Delta) e^{-i \tau \frac{1}{2}\Delta }\overline{u_{\ast,\infty}^{n}}
\Big]
- \frac{9\tau^2}{2\cdot 64}   \vert u_{\ast,\infty}^{n}\vert^4   {u_{\ast,\infty}^{n}}.
\end{aligned}
\end{equation}

\end{rem}

\section{Numerical Experiments}\label{sec:num}
In this section we numerically underline our theoretical findings. In particular we observe the uniform accuracy in $c$  and low regularity approximation of the schemes as stated in Theorem \ref{thm:glob} and Theorem \ref{thm:glob2}.

{ In the convergence order experiments we approximate the spatial differential operators via a standard Fourier pseudo-spectral method, choosing $M=200$ as the highest Fourier mode, which corresponds to $\Delta x=0.005$. We considered randomly generated $H^1$ and $H^2$ initial data respectively (see for instance \cite{OS18}) and we integrate them up to $t=1$. The step size $\tau$ we have chosen to be $2^{-i}$, where $i=6,\dots,16$.}

In Figure \ref{1st_order_plot} we plot the convergence of the first order scheme \eqref{scheme} for various values of $c$ for a rough initial data $u(0) \in H^1$. In Figure \ref{2nd_order_plot} we plot the convergence of the second order scheme \eqref{scheme2} for various values of $c$ for a rough initial data $u(0) \in H^2$.   For the spatial discretisation we use a classical Fourier pseudo spectral method. Our numerical findings underline the uniform accuracy of the schemes  for rough data.
 \begin{figure}[h!]
  \includegraphics[width=0.45\textwidth]{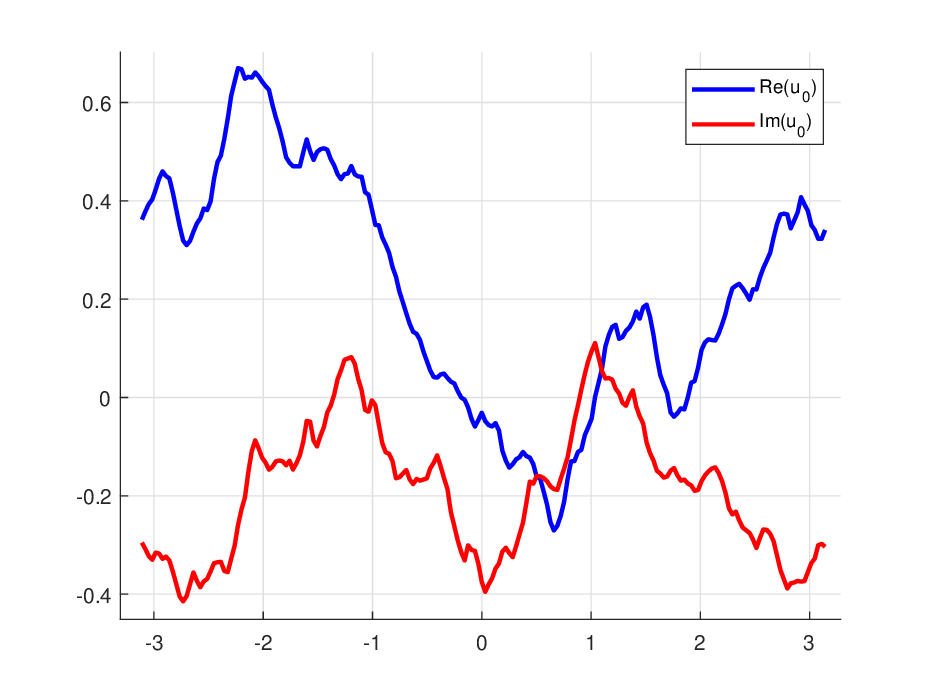}
  \includegraphics[width=0.45\textwidth]{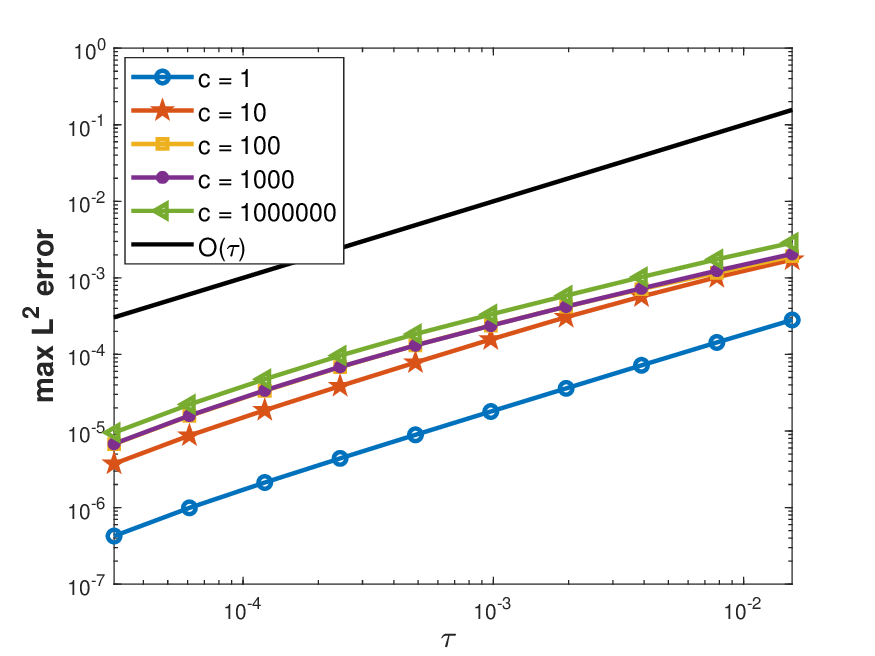}
  \caption{Convergence plot of the first order scheme \eqref{scheme} with $H^1$ initial data (left).}
  \label{1st_order_plot}
\end{figure}
 \begin{figure}[h!]
  \includegraphics[width=0.45\textwidth]{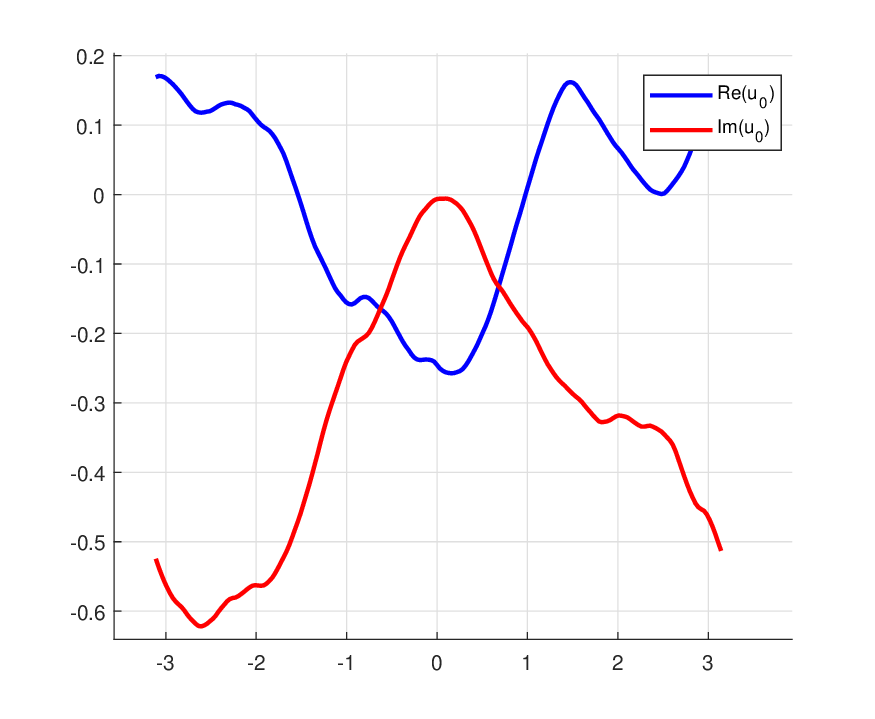}
  \includegraphics[width=0.45\textwidth]{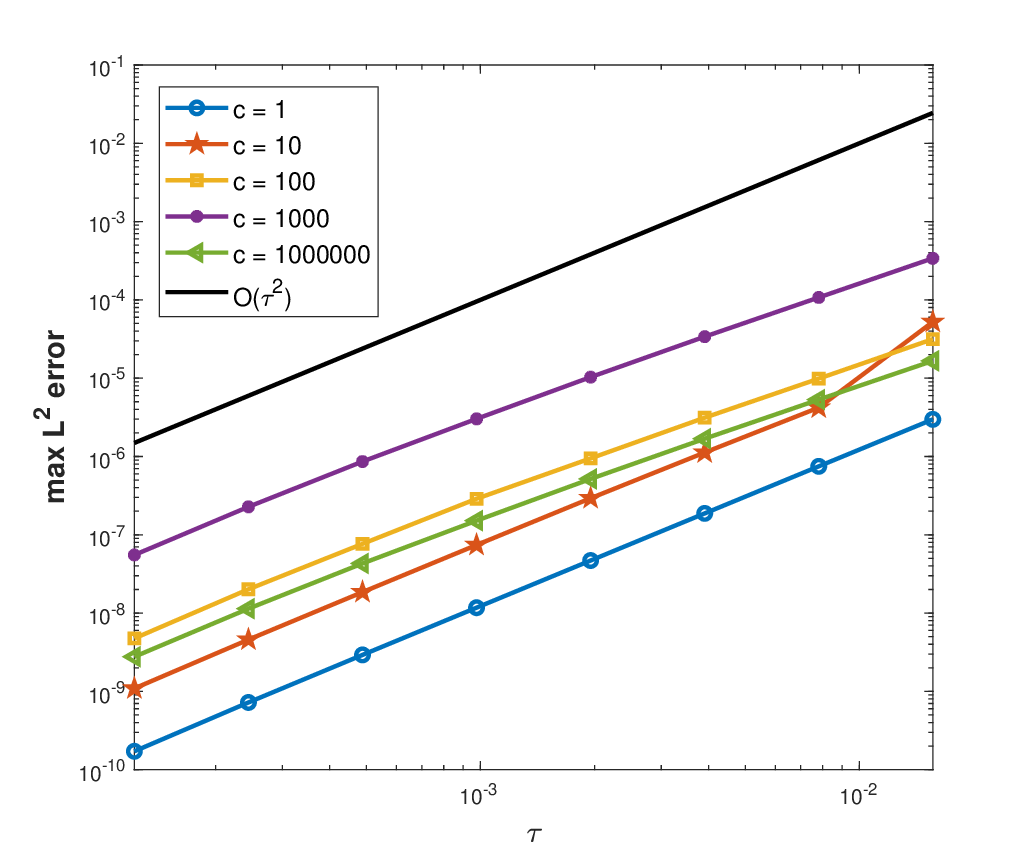}
  \caption{Convergence plot of the second order scheme \eqref{scheme2} with $H^2$ initial data (left).}
  \label{2nd_order_plot}
\end{figure}

{ In this last numerical experiment, we employed the same spatial approximation method and we integrate $H^2$ initial data up to $t=100$ using our second order method. We show the difference in absolute value of the energy at time $0<t\leq 100$ and at time $t=0$.}

\begin{figure}[h!]
  \includegraphics[width=0.32\textwidth]{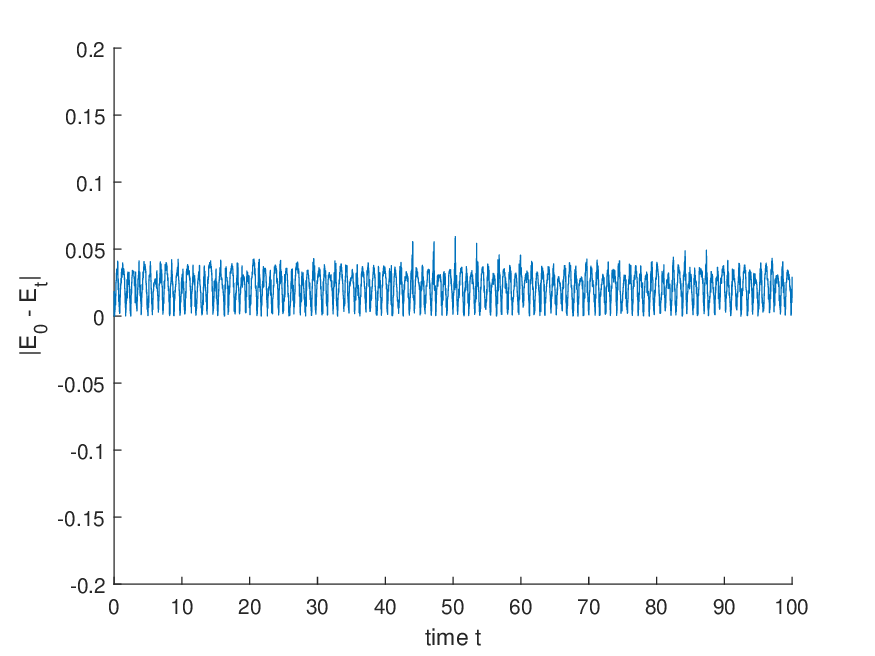}
  \includegraphics[width=0.32\textwidth]{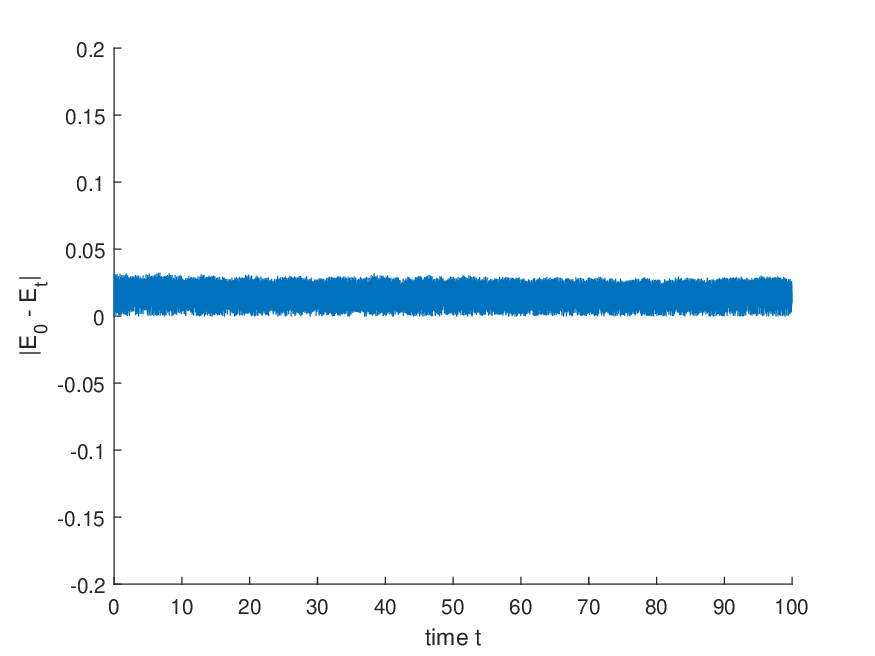}
  \includegraphics[width=0.32\textwidth]{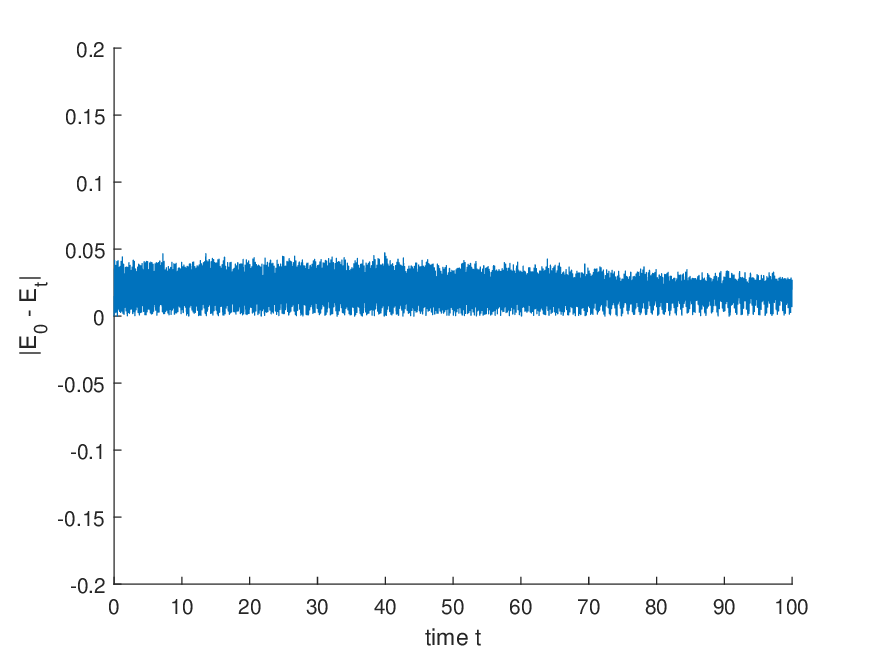}
  \caption{{Plot showing the difference between the initial energy and the energy at time $t$, up to $t=100$ for the values $c=1$, $c=10$ and $c=100$ respectively.}}
  \label{1st_order_plot}
\end{figure}

\subsection*{Acknowledgements}

{\small
The authors have received funding from the European Research Council (ERC) under the European Union’s Horizon 2020 research and innovation programme (grant agreement No. 850941).
}

\end{document}